\documentclass[11pt]{article}

\newtheorem{theorem}{Theorem}[section]
\newtheorem{lemma}[theorem]{Lemma}
\newtheorem{proposition}[theorem]{Proposition}
\newtheorem{corollary}[theorem]{Corollary}

\newenvironment{proof}[1][Proof]{\begin{trivlist}
		\item[\hskip \labelsep {\bfseries #1}]}{\end{trivlist}}
\newenvironment{definition}[1][Definition]{\begin{trivlist}
		\item[\hskip \labelsep {\bfseries #1}]}{\end{trivlist}}

\newenvironment{remark}[1][Remark]{\begin{trivlist}
		\item[\hskip \labelsep {\bfseries #1}]}{\end{trivlist}}

\newcommand{\qed}{\nobreak \ifvmode \relax \else
	\ifdim\lastskip<1.5em \hskip-\lastskip
	\hskip1.5em plus0em minus0.5em \fi \nobreak
	\vrule height0.75em width0.5em depth0.25em\fi}
\usepackage[margin=1in,includefoot]{geometry}

\usepackage{setspace}
\usepackage{comment}
\usepackage{tikz}
\usetikzlibrary{calc,patterns}
\usepackage{tikz-network}

\usepackage{cancel}

\usepackage{placeins}

\usepackage[utf8]{inputenc}

\usepackage{blindtext}
\usepackage{todonotes}

\usepackage{graphicx}
\usepackage{float}

\usepackage{xfrac}

\usepackage{fancyhdr}

\usepackage{mathtools}
\usepackage{nccmath}

\usepackage{relsize}

\usepackage{bm}

\usepackage{titling}
\pretitle{\begin{center}\Huge\bfseries}
	\posttitle{\par\end{center}\vskip 0.5em}
\preauthor{\begin{center}\Large\ttfamily}
	\postauthor{\end{center}}
\predate{\par\large\centering}
\postdate{\par}

\title{A Structural Characterization of Determinantally Equivalent Functions}
\author{
	Harry Sapranidis Mantelos\\ [1mm]
	\small The Hebrew University of Jerusalem\\
	\small \texttt{Charalampos.Sapranidis-mantelos@mail.huji.ac.il}
}
\date{\today}

\usepackage{bbm}

\usepackage{relsize}

\usepackage{mathtools}
\usepackage{csquotes}
\usepackage{mathrsfs}
\usepackage{amsbsy}

\usepackage{amssymb}

\usepackage{enumerate}
\usepackage{url}

\newcommand{\vertiii}[1]{{\left\vert\kern-0.25ex\left\vert\kern-0.25ex\left\vert #1 
		\right\vert\kern-0.25ex\right\vert\kern-0.25ex\right\vert}}

\allowdisplaybreaks
\begin{document}
	
	\maketitle
	\thispagestyle{empty}

	\begin{abstract}
		Let $\Lambda$ be a set and $\mathbb{F}$ a field. Suppose that $K,Q:\Lambda^2\to\mathbb{F}$ are two functions such that for any $n\in\mathbb{N}$ and $x_1,x_2,\ldots,x_n\in\Lambda$, the determinants of matrices $(K(x_i,x_j))_{1\leq i,j\leq n}$ and $(Q(x_i,x_j))_{1\leq i,j\leq n}$ agree. We study to what extent $K$ and $Q$ must be related by two canonical transformations corresponding to diagonal similarity and transposition.
		
		In the symmetric case, this relation holds without further assumptions (see Stevens (2021), \cite{equiv_symm_kernels_for_dpps}), while in general it fails. In Mantelos (2026), \cite{mantelos2026determinantally}, it was shown that the relation remains valid under a natural $2\times 2$ determinantal condition (property $\mathcal{D}$), together with the additional assumption that both functions are nowhere vanishing.
		
		We prove that the `nowhere vanishing' assumption can be removed entirely, and that property $\mathcal{D}$ alone provides the correct and complete structural mechanism governing the problem. In particular, this shows that the nowhere-zero assumption is not intrinsic to the problem, but rather an artefact of the specific method used in \cite{mantelos2026determinantally}.
		
		The proof is entirely combinatorial and avoids linear algebra, relying on an analysis of permutations in the definition of a determinant as cycles in a graph; in particular, it requires new arguments to handle the breakdown of the identities used in \cite{mantelos2026determinantally}, which are crucial to the method therein. In the `finite $\Lambda$' case, this also yields a new approach to the classical matrix problem of Loewy (1986), \cite{loewy1986principal}, thereby revealing an underlying combinatorial structure.
	\end{abstract}
	\tableofcontents
	\section{Introduction and Main Result}
	
	Let $\Lambda$ be an abstract set (of arbitrary cardinality) and let $\mathbb{F}$ be an arbitrary field. Suppose that $K,Q:\Lambda^2\to\mathbb{F}$ are two \textit{determinantally equivalent} functions, that is,
	\begin{equation}
		\label{equal_determinants_relation}
		\det(Q(x_i,x_j))_{i,j=1}^n=\det(K(x_i,x_j))_{i,j=1}^n\qquad \forall x_1,\ldots,x_n\in\Lambda\enspace\forall n\in\mathbb{N}.
	\end{equation}
	We study to what extent such functions must be related by the following two canonical transformations:
	\begin{enumerate}
		\item \textbf{Conjugation transformations}, where
		\begin{equation}
			\label{conjugation_transformation}
			Q(x,y) = g(x) K(x,y) g(y)^{-1}\qquad \forall x,y\in\Lambda,
		\end{equation}
		for some nowhere zero function $g:\Lambda\to\mathbb{F}$; and
		\item \textbf{Transposition transformations}, where
		\begin{equation}
			\label{transposition_transformation}
			Q(x,y)= K(y,x)\qquad \forall x,y\in\Lambda.
		\end{equation}
	\end{enumerate}
	It is clear that both of the above transformations preserve \eqref{equal_determinants_relation}; the paper therefore focuses on the extent to which \eqref{equal_determinants_relation} implies \eqref{conjugation_transformation} and/or \eqref{transposition_transformation}. In the case where the underlying `set $\Lambda$' is finite, this problem actually reduces to the classical question in linear algebra concerning principal minors and diagonal similarity of matrices; cf. \cite{loewy1986principal}. In this setting, the present paper provides a new combinatorial approach to this classical matrix problem. The functional formulation was first introduced in \cite{equiv_symm_kernels_for_dpps}, where it was motivated by determinantal point processes. However, both in that work and in the present paper, the problem is formulated and studied in a purely deterministic setting, independent of any probabilistic interpretation.
	
	In \cite{equiv_symm_kernels_for_dpps}, the problem was studied under the additional assumption that $K$ and $Q$ are symmetric, in which case a transposition transformation plays no role. More precisely, it was shown that determinantally equivalent symmetric functions are necessarily related by conjugation transformations. This extends, in the finite setting (i.e., in the setting where the underlying `set $\Lambda$' is finite), a result for (real or complex) symmetric matrices with equal corresponding principal minors; cf. \cite{kulesza2012learning}. Related extensions to (complex) Hermitian matrices can be found in \cite{launay2021determinantal}.
	
	The more general -- not necessarily symmetric -- setting is more subtle. As explained in \cite{mantelos2026determinantally}, without further assumptions on $K$ and $Q$ the conclusion fails in general, as demonstrated by standard counterexamples involving block matrices where only certain blocks get transposed. To preclude such pathologies, a natural structural condition -- referred to as property $\mathcal{D}$ -- was introduced:
	
	\begin{definition}
		We say that a function $h:\Lambda^2\to\mathbb{F}$ is of \textit{class $\mathcal{D}$} if for every pairwise distinct $x,y,z,w\in\Lambda$,
		\begin{equation}
			\label{property_D}
			\begin{vmatrix}
				h(x,y) & h(x,z) \\
				h(w,y) & h(w,z)
			\end{vmatrix} \neq 0.
		\end{equation}
	\end{definition}
	
	Roughly speaking, property $\mathcal{D}$ imposes a non-degeneracy condition that prevents the kind of local symmetries that give rise to counterexamples; e.g., $\Lambda=\{1,2,\ldots,n\}$ and functions/matrices $K,Q:\Lambda^2\to\mathbb{F}$ such that
	
	\[(K(x,y))_{1\leq x,y\leq n}=
	\renewcommand\arraystretch{1.5}
	\left[\begin{array}{@{}c|c@{}}
		\mbox{\Large C} & \mbox{\Large 0} \\
		\hline
		\mbox{\Large 0} & \mbox{\Large D}
	\end{array}\right], \qquad (Q(x,y))_{1\leq x,y\leq n}=\left[\begin{array}{@{}c|c@{}}
		\mbox{\Large $C^{T}$} & \mbox{\Large $0$} \\
		\hline
		\mbox{\Large $0$} & \mbox{\Large $D$}
	\end{array}\right],
	\]
	
	in which case $K$ and $Q$ \textit{are} determinantally equivalent but are \textit{not} related by either conjugation transformations or transposition transformations. Instead, $K$ and $Q$ are related through this type of `\textit{partial transposition transformation}' where only one block gets transposed. Condition $\mathcal{D}$ rules out such counterexamples as it requires non-zero determinants in the upper-right and lower-left sub-matrix regions. 
	
	Under condition $\mathcal{D}$, it was shown in \cite{mantelos2026determinantally} that if $K$ and $Q$ are determinantally equivalent and nowhere-vanishing up to the set $\{(x,x):x\in\Lambda\}$, then they are in fact related by conjugation and transposition transformations.
	
	The presence of the nowhere-zero assumption raises a natural question: is it an essential part of the phenomenon, or merely a technical device used to facilitate/enable the specific argument in \cite{mantelos2026determinantally}? The main result of this paper shows that it is not essential -- property $\mathcal{D}$ alone provides the underlying structural mechanism.
	
	%Under condition $\mathcal{D}$, it was proven in \cite{mantelos2026determinantally} that if $K$ and $Q$ are determinantally equivalent and nowhere-vanishing up to the set $\{(x,x):x\in\Lambda\}$, then they are related by conjugation and transposition transformations. We remark that property $\mathcal{D}$ was not fully exploited in \cite{mantelos2026determinantally}, where it played a role only at the very final stage of the argument. This suggests that its true structural significance is not fully captured by that paper and goes far beyond its original use. The purpose of the present paper is to make this precise: we show that the nowhere-zero assumption is not essential, and that property $\mathcal{D}$ alone provides the correct structural framework. More precisely, our main result is thus.
	
	\begin{theorem}
		\label{main_thm}
		Let $\Lambda$ be a set, $\mathbb{F}$ a field, and let $K,Q:\Lambda^2\to\mathbb{F}$ be two functions. Suppose further that $K$ and $Q$ are of class $\mathcal{D}$. If $K$ and $Q$ are determinantally equivalent, then $Q$ can be transformed into $K$ through only conjugation and transposition transformations.
	\end{theorem}
	
	We remark further that property $\mathcal{D}$ was not fully exploited in \cite{mantelos2026determinantally}, where it appeared only at the final stage of the argument. This suggests that its true structural significance is not fully captured by that paper and goes far beyond its original use. The present work makes its structural role precise.
	
	While the overall strategy is inspired by \cite{mantelos2026determinantally}, the techniques developed here are substantially different. In particular, the identities used in that work -- central to the analysis -- are no longer applicable in the presence of zero values. As in that paper, and in contrast to classical approaches to Loewy’s problem \cite{loewy1986principal} relying on linear-algebraic machinery, our approach is entirely combinatorial: the proof is based on analyzing permutations in the definition of the determinant as cycles in a graph. From this cycle structure, we naturally construct two explicit bivariate functions $S$ and $\tilde{S}$ (given in \eqref{eq:def_S} and \eqref{eq:def_S_tilde}, respectively), which capture the dichotomy between the two structural cases of the problem. The main step is to establish that all $3$-cycles fall into a single category, namely either Case 1 or Case 2 in the terminology of \cite{mantelos2026determinantally}. Under this classification, the argument splits according to the two cases: In the Case 1 situation, we show that $S$ satisfies the cocycle property (defined precisely in Section~\ref{sec_notations_graph_id}), thereby yielding conjugation equivalence \eqref{conjugation_transformation}; and in the Case 2 situation, $\tilde{S}$ is shown to be a cocycle function, thereby yielding transposition equivalence \eqref{transposition_transformation}, up to conjugation. A central part of the argument is a careful analysis of the constraints imposed by property $\mathcal{D}$ on $3$-cycles, which ultimately reveals it as the genuine structural condition governing the problem.

	The paper is organized as follows. In Section~\ref{sec_notations_graph_id}, besides recalling the notations and terminology from \cite{mantelos2026determinantally}, we introduce several more together with a graphical framework adapted to the presence of vanishing values. Section~\ref{sec_basic_results} recalls background results from \cite{equiv_symm_kernels_for_dpps} and \cite{mantelos2026determinantally}. Section~\ref{sec_strategy} outlines the strategy of our proof of Theorem~\ref{main_thm}. Section~\ref{sec_further_identities} develops new cycle-based identities capable of handling vanishing values, which may be viewed as extensions of certain identities from Section 5 of \cite{mantelos2026determinantally}, to the setting with zeros. Section~\ref{sec_implications_on_D_class_assumption} analyzes the implications of the `class $\mathcal{D}$' assumption, establishing several structural properties crucial to our arguments later on. Finally, utilizing all the tools developed in the previous sections, Section~\ref{sec_proof_of_main_result} contains our proof of Theorem~\ref{main_thm}.
	
	\section*{Acknowledgements}
	This research was supported by the Israel Science Foundation grant 305/25.
	
	\section{Preliminaries}
	\subsection{Notation \& Graphical Framework}
	\label{sec_notations_graph_id}
	
	Throughout the paper, we adopt the notation and terminology introduced in \cite{mantelos2026determinantally}. For the convenience of the reader, we briefly recall here the main definitions and conventions that will be used. Afterwards, we introduce some additional notations adapted to accommodate the presence of zero values in the codomain of the functions under consideration. 
	
	This additional level of generality necessitates new tools for tracking structural behaviour. To this end, we also introduce a graphical framework that allows us to represent diagrammatically the vanishing and non-vanishing of $K(x,y)$ across pairs $(x,y)\in\Lambda^2$. This new graphical identification will play a key role in clarifying and guiding several of the arguments developed later in the paper.
	
	\begin{definition}
		Let $\mathcal{M}$ be a (possibly infinite) set. For an integer $1\leq n\leq |\mathcal{M}|$, we define a \textit{cycle of length $n$ in $\mathcal{M}$} (or an \textit{$n$-cycle in $\mathcal{M}$}, for short) to be an $(n+1)$-tuple of the form $p=(p_0,p_1,\ldots,p_n)\in\mathcal{M}^{n+1}$, where each of the $p_i$ are distinct except for $p_0=p_n$. We will refer to the $p_i$ as the \textit{vertices of the cycle $p$}.
	\end{definition}
	
	\begin{definition}
		Let $\mathcal{M}$ be a (possibly infinite) set and fix an integer $1\leq n\leq |\mathcal{M}|$. Given an $n$-cycle $p=(p_0,p_1,\ldots,p_n)$ in $\mathcal{M}$, we denote by $p'\coloneqq (p_n,p_{n-1},\ldots,p_0)$ the \textit{cycle $p$ in reverse}.
	\end{definition}
	
	\begin{definition}
		Let $\Lambda$ be a (possibly infinite) set and let $\mathbb{F}$ be a field. Fix an integer $1\leq n\leq|\Lambda|$. For a bivariate function $h:\Lambda^2\to\mathbb{F}$ and an $n$-cycle, $p=(p_0,\ldots,p_n)$, in $\Lambda$, we denote by $h[p]$ the product
		\begin{equation*}
			h[p]\coloneqq\prod_{i=1}^n h(p_{i-1},p_i);
		\end{equation*}
		and by $h'[p]$ the analogous product with respect to the reverse cycle $p'$:
		\begin{equation*}
			h'[p]\coloneqq\prod_{i=1}^n h(p_i,p_{i-1}) \enspace(=h[p']).
		\end{equation*}
	\end{definition}
	
	Next, we recall a central definition regarding $3$-cycles that was introduced in \cite{mantelos2026determinantally}.
	
	\begin{definition}
		Let $\Lambda$ be a set, $\mathbb{F}$ a field, and let $K,Q:\Lambda^2\to\mathbb{F}$ be two functions. We say that a $3$-cycle, $p=(p_i)_{i=0}^3$, in $\Lambda$ belongs to Case 1 or Case 2 if
		\begin{equation*}
			\text{\textbf{Case 1:} } K[p]=Q[p] \text{ and } K'[p]=Q'[p];
		\end{equation*}
		or
		\begin{equation*}
			\text{\textbf{Case 2:} } K[p]=Q'[p] \text{ and } K'[p]=Q[p],
		\end{equation*}
		respectively.
	\end{definition}
	
	\begin{comment}	
		\begin{definition}
			Let $\Lambda$ be a set, $\mathbb{F}$ a field, and let $K,Q:\Lambda^2\to\mathbb{F}$ be two functions. Let $p=(p_i)_{i=0}^3$ be a $3$-cycle in $\Lambda$.
			\begin{enumerate}[\itshape(i)]
				\item Suppose $p$ belongs to Case 1. We refer to an edge $(p_{j-1},p_j)$ of $p$ as a \textit{zero-edge} if $K(p_{j-1},p_j)=0$; and we refer to it as a \textit{nonzero-edge} if and only if $K(p_{j-1},p_j)\neq0$.
				\item Suppose $p$ belongs to Case 2. We refer to an edge $(p_{j-1},p_j)$ of $p$ as a \textit{zero-edge} if and only if $K(p_j,p_{j-1})=0$; and we refer to it as a \textit{nonzero-edge} if $K(p_j,p_{j-1})\neq0$.
			\end{enumerate}
		\end{definition}
		
		\begin{remark}
			The above definition is well-defined. Indeed, let $p$ and $q$ be two distinct $3$-cycles in $\Lambda$ that share a common edge, say $(x,y)\in\Lambda^2$. This means that $p=(p_*,x,y,p_*)$ and $q=(q_*,x,y,q_*)$ for some $p_*,q_*\in\Lambda$ distinct. Now suppose that $K(x,y)=0$ and that $p$ belongs to Case 1 and $q$ to Case 2, in which case $(x,y)$ is a zero-edge according to the previous definition.  
		\end{remark}
	\end{comment}
	
	Finally, we introduce and explain the graphical framework mentioned earlier that aids in keeping track of the zero and non-zero values of the function $K$: Fix $(x,y)\in\Lambda^2$.
	\begin{itemize}
		\item If $K(x,y)=0$, we label the edge $(x,y)$ with `$0$'.
		\item If $K(x,y)\neq 0$, we label the edge $(x,y)$ with `$\neg 0$'.
	\end{itemize}
	For instance, suppose we were analyzing the $3$-cycle $p=(p_i)_{i=0}^3$ in $\Lambda$ and it were the case that $K(p_0,p_1)=0$, $K(p_1,p_2)=0$ and $K(p_2,p_0)\neq 0$. Diagrammatically we would represent this setup as follows:
	\begin{figure}[H]
		\centering
		\begin{tikzpicture}
			\Vertex[label=$p_2$]{A} \Vertex[x=6,label=$p_1$]{B}, \Vertex[x=3,y=3,label=$p_0$]{C} 
			\Edge[Direct,color=black,label=$\neg 0$](A)(C)
			\Edge[Direct,color=black,label=0](C)(B)
			\Edge[Direct,color=black,label=0](B)(A)
		\end{tikzpicture}
	\end{figure}
	
	Finally, we recall the labelling of the $4$-cycles and $3$-cycles in $\mathcal{M}\coloneqq\{1,2,3,4\}$ from \cite{mantelos2026determinantally}:
	\begin{equation}
		\label{four_cycle_labellings}
		q^{[1]}\coloneqq (1,2,3,4,1),\qquad q^{[2]}\coloneqq (1,2,4,3,1),\qquad q^{[3]}\coloneqq (1,3,2,4,1).
	\end{equation}
	\begin{equation}
		\label{three_cycle_labellings}
		p^{(1)}\coloneqq (1,2,3,1),\qquad p^{(2)}\coloneqq (1,2,4,1),\qquad p^{(3)}\coloneqq (1,4,3,1),\qquad p^{(4)}\coloneqq (2,3,4,2).
	\end{equation}
	We will be consistent with the above labelling of cycles throughout the paper.
	
	\subsection{Background results}
	\label{sec_basic_results}
	
	In this section, we recall the necessary background results that are used throughout the paper, taken from \cite{mantelos2026determinantally} and \cite{equiv_symm_kernels_for_dpps}.
	
	There are two simple -- but very useful -- relations between determinantally equivalent functions arising from equation (\ref{equal_determinants_relation}) with $n=1,2$, that were used extensively in \cite{mantelos2026determinantally}. We collect them in the following lemma. 
	
	\begin{lemma}
		\label{lemma_det_equivalence_w_easy_consequence}
		Let $K,Q:\Lambda^2\to\mathbb{F}$ be a pair of determinantally equivalent functions. Then,
		\begin{enumerate}[(i)]
			\item $K(x,x)=Q(x,x)$ for every $x\in\Lambda$.
			\item $K(x,y)K(y,x)=Q(x,y)Q(y,x)$ for every $x,y\in\Lambda$.
		\end{enumerate}
	\end{lemma}
	\begin{proof}
		This is an easy and direct consequence of equation (\ref{equal_determinants_relation}) with $n=1$ and $n=2$. \qed
	\end{proof}
	
	Next, we recall the result that establishes the Case 1/Case 2 classification for $3$-cycles in $\Lambda$ (recall the respective definition from Section~\ref{sec_notations_graph_id}):
	
	\begin{lemma}[Lemma 6.2, \cite{mantelos2026determinantally}]
		\label{prop_about_case1_case2_for_3_cycles}
		Let $\Lambda$ be a set, and $\mathbb{F}$ a field. If functions $K,Q:\Lambda^2\to\mathbb{F}$ are determinantally equivalent, then every $3$-cycle, $p=(p_i)_{i=0}^3$, in $\Lambda$ belongs to at least one of Case 1 or Case 2.
	\end{lemma}
	\begin{proof}
		Equation (\ref{equal_determinants_relation}) with $n=3$ in conjunction with the Leibniz formula for determinants yields
		\begin{equation*}
			\sum_{\sigma\in S_3} \text{sgn}(\sigma) \prod_{i=1}^3 K(p_i,p_{\sigma(i)}) = \sum_{\sigma\in S_3} \text{sgn}(\sigma) \prod_{i=1}^3 Q(p_i,p_{\sigma(i)}).
		\end{equation*}
		By utilizing Lemma~\ref{lemma_det_equivalence_w_easy_consequence}, one then arrives at
		\begin{equation}
			\label{key_trick_mythm_eq1}
			K[p]+K'[p]=Q[p]+Q'[p].
		\end{equation}
		But, thanks to the second part of Lemma~\ref{lemma_det_equivalence_w_easy_consequence}, we also have
		\begin{equation}
			\label{key_trick_mythm_eq2}
			K[p]K'[p]=Q[p]Q'[p].
		\end{equation}
		Equations (\ref{key_trick_mythm_eq1}) and (\ref{key_trick_mythm_eq2}) together prove the claim of the lemma. \qed
	\end{proof}
	
	\begin{remark}
		Notice how in the definition for Case 1/Case 2 cycles from Section~\ref{sec_notations_graph_id}, if $K$ and $Q$ are determinantally equivalent, then the first equality implies the second and vice versa. It is important to note that this holds regardless of whether $K$ or $Q$ is nowhere zero, and follows immediately from \eqref{key_trick_mythm_eq1}. Therefore, establishing just the first (or the second) equality for a $3$-cycle $p$ is sufficient; and this fact is implicitly used throughout the paper.
	\end{remark}
	
	We recall a basic result for $3$-cycles that was used implicitly in \cite{mantelos2026determinantally}.
	
	\begin{lemma}
		\label{lemma_3cycle_in_both_cases_condition}
		Let $\Lambda$ be a set, and $\mathbb{F}$ a field. Let $p$ be a $3$-cycle in $\Lambda$. If functions $K,Q:\Lambda^2\to\mathbb{F}$ are determinantally equivalent, then
		\begin{equation*}
			\text{$p$ belongs to both Case 1 and Case 2}\iff K[p]=K'[p]\iff Q[p]=Q'[p].
		\end{equation*} 
	\end{lemma}
	\begin{proof}
		This is an immediate consequence of Lemma~\ref{prop_about_case1_case2_for_3_cycles}. \qed
	\end{proof}
	
	Next, we recall the definition of another important -- for our purposes -- type of function and transformation.
	\begin{definition}
		Let $\Lambda$ be a (possibly infinite) set and let $\mathbb{F}$ be a field. Fix an integer $1\leq n\leq|\Lambda|$. We say that a two-variable function $c:\Lambda^2\to\mathbb{F}$ satisfies the \textit{cocycle property for $n$-cycles} if for every $n$-cycle $p$ in $\Lambda$, $c[p]=1$. 
		
		If $c$ satisfies the cocycle property for cycles in $\Lambda$ of every length $n$, where $1\leq n \leq |\Lambda|$, we say that $c$ is a (full) \textit{cocycle function}. Equivalently, $c$ is a cocycle function if for all $z,w\in\Lambda$, $c(z,z)=1$ and $c(z,w)c(w,z)=1$, and for all integers $r>2$ and every $r$-tuple $(z_1,\ldots,z_r)\in\Lambda^r$,
		\begin{equation}
			\label{cocycle_property}
			c(z_1,z_2)c(z_2,z_3)\cdots c(z_{r-1},z_r)c(z_r,z_1)=1.
		\end{equation}
	\end{definition}
	
	\begin{definition}
		If $K,Q:\Lambda^2\to\mathbb{F}$ are functions that satisfy
		\begin{equation*}
			Q(x,y)=c(x,y)K(x,y)\qquad\forall x,y\in\Lambda
		\end{equation*}
		for some cocycle function $c:\Lambda^2\to\mathbb{F}$, we say that $Q$ is a \textit{cocycle transformation} of $K$.
	\end{definition}
	
	\begin{remark}
		In Proposition~4.3 from \cite{mantelos2026determinantally} it was shown that to prove that a function $c:\Lambda^2\to\mathbb{F}$ is a cocycle function, it suffices to establish the cocycle property just for cycles in $\Lambda$ of lengths $1$, $2$ and $3$, that is, it suffices to prove 
		\begin{enumerate}[\itshape(i)]
			\item{$c(x,x)=1$ for every $x\in\Lambda$;} \label{property_i_of_suff_cocycle}
			\item{$c(x,y)c(y,x)=1$ for every $x,y\in\Lambda$;} \label{property_ii_of_suff_cocycle}
			\item{$c(x,y)c(y,z)c(z,x)=1$ for every $x,y,z\in\Lambda$,} \label{property_iii_of_suff_cocycle}
		\end{enumerate}
	\end{remark}
	
	\begin{remark}
		We end this section by recalling the following fact, taken from Proposition~4.2 of \cite{mantelos2026determinantally}, regarding two functions $K,Q:\Lambda^2\to\mathbb{F}$: 
		
		$Q$ is a conjugation transformation of $K$ $\iff$ $Q$ is a cocycle transformation of $K$.
	\end{remark}
	
	\section{Proof Strategy}
	\label{sec_strategy}
	This section outlines the strategy of our derivation of Theorem~\ref{main_thm}.
	
	Thanks to the final remark from the previous section, to prove Theorem~\ref{main_thm}, it suffices to construct cocycle functions $S,\tilde{S}:\Lambda^2\to\mathbb{F}$ such that either
	\begin{equation}
		\label{eq_impossible_1}
		Q(z,w)=S(z,w)K(z,w)\qquad\forall z,w\in\Lambda,
	\end{equation} 
	or
	\begin{equation}
		\label{eq_impossible_2}
		Q(z,w)=\tilde{S}(z,w)K(w,z)\qquad\forall z,w\in\Lambda.
	\end{equation} 
	holds.
	
	We first note that thanks to Lemma~\ref{lemma_det_equivalence_w_easy_consequence} (i), we can choose for every $x\in\Lambda$, $$S(x,x)=1=\tilde{S}(x,x).$$
	
	Now, in order for either \eqref{eq_impossible_1} or \eqref{eq_impossible_2} to hold, we would first need to guarantee that either
	\begin{equation}
		\label{case1_relation}
		Q(z,w)=0\iff K(z,w)=0\qquad\forall z,w\in\Lambda\text{ distinct}
	\end{equation}
	or
	\begin{equation}
		\label{case2_relation}
		Q(z,w)=0\iff K(w,z)=0\qquad\forall z,w\in\Lambda\text{ distinct}
	\end{equation}
	is true, respectively. Provided that all $3$-cycles in $\Lambda$ belong to the same Case (which we establish in Section~\ref{sec_proof_of_main_result}), we will see in Section~\ref{sec_implications_on_D_class_assumption} -- and more specifically Proposition~\ref{prop_guarantee_nonexistence_of_problematic_pairs} -- that under our `class $\mathcal{D}$' hypothesis on $K$ and $Q$, one of \eqref{case1_relation} or \eqref{case2_relation} must indeed hold.
	
	Of course, in the situation where for every $z,w\in\Lambda$ distinct $K(z,w)\neq 0$ (equivalently, $Q(z,w)\neq 0$, by Lemma~\ref{lemma_det_equivalence_w_easy_consequence} (ii)), the problem reduces to that of \cite{mantelos2026determinantally}, where it was proven that at least one of 
	$$S(x,y)=\begin{cases*}
		\frac{Q(x,y)}{K(x,y)}, &\text{if $x\neq y$} \\
		1, &\text{if $x=y$}
	\end{cases*}\qquad\text{or}\qquad\tilde{S}(x,y)=\begin{cases*}
		\frac{Q(x,y)}{K(y,x)}, &\text{if $x\neq y$} \\
		1, &\text{if $x=y$}
	\end{cases*},\qquad x,y\in\Lambda,$$
	is a cocycle function.
	
	In the presence of zero values of $K$ and $Q$, the above definitions of $S$ and $\tilde{S}$ of course need to be modified. In particular, in Section~\ref{sec_proof_of_main_result} we prove that at least one of the following two functions is a cocycle function under the `class $\mathcal{D}$' hypothesis.
	\begin{equation}
		\label{eq:def_S}
		S(x,y)=
		\begin{cases*}
			1, & if $x=y$,\\
			\dfrac{Q(x,y)}{K(x,y)}, & if $K(x,y)\neq 0$,\\		
			\left(\dfrac{Q(y,x)}{K(y,x)}\right)^{-1},
			& \parbox[t]{.55\textwidth}{if $K(x,y)=0$ and $(x,y)$ is an edge of a $3$-cycle belonging only to Case 1,}\\
			\dfrac{Q(x,z)Q(z,y)}{K(x,z)K(z,y)},
			& \parbox[t]{.55\textwidth}{if $K(x,y)=0$ and $(x,y)$ is an edge of a $3$-cycle $(z,x,y,z)$ belonging to both Cases 1 and 2}
		\end{cases*}
	\end{equation}
	and
	\begin{equation}
		\label{eq:def_S_tilde}
		\tilde{S}(x,y)=
		\begin{cases*}
			1, & if $x=y$,\\
			\dfrac{Q(x,y)}{K(y,x)}, & if $K(y,x)\neq 0$,\\
			\left(\dfrac{Q(y,x)}{K(x,y)}\right)^{-1},
			& \parbox[t]{.55\textwidth}{if $K(y,x)=0$ and $(x,y)$ is an edge of a $3$-cycle belonging only to Case 2,}\\
			\dfrac{Q(x,z)Q(z,y)}{K(z,x)K(y,z)},
			& \parbox[t]{.55\textwidth}{if $K(y,x)=0$ and $(y,x)$ is an edge of a $3$-cycle $(z,y,x,z)$ belonging to both Cases 1 and 2}
		\end{cases*}
	\end{equation}
	We will delay answering questions pertaining to the well-definedness of the above two functions once we have established several key results to do with the effect that the `class $\mathcal{D}$' hypothesis has on $3$-cycles in $\Lambda$ (cf. Section~\ref{sec_implications_on_D_class_assumption}). Specifically, we will address the well-definedness of the two functions in Lemma~\ref{lemma_welldefined}.
	
	By combining the key structural results established in Proposition~\ref{prop_guarantee_nonexistence_of_problematic_pairs}, Lemma~\ref{lemma_welldefined} and Proposition~\ref{prop_cocycle_property_for_S_tildeS}, we obtain the following key result: if every $3$-cycle in $\Lambda$ belongs to Case 1 (resp. Case 2), then \eqref{case1_relation} and \eqref{eq_impossible_1} (resp. \eqref{case2_relation} and \eqref{eq_impossible_2}) hold and $S$ (resp. $\tilde{S}$) is a cocycle function. 
	
	Consequently, the proof of Theorem~\ref{main_thm} then reduces to showing that all $3$-cycles in $\Lambda$ belong to Case 1, or all to Case 2. Establishing this is the main objective of Section~\ref{sec_proof_of_main_result}, and in particular of Proposition~\ref{prop_step_2_lambda_4_elements} together with Lemmas~\ref{prop_step_2.1_lambda_4_elements} and \ref{prop_step_2.1_lambda_4_elements2}. 
	
	We remark that to establish this classification it suffices to prove it for subsets $\mathcal{M}\subseteq \Lambda$ with $|\mathcal{M}|=4$. Equivalently, it is enough to show that the four $3$-cycles listed in \eqref{three_cycle_labellings} belong to the same Case. Indeed, as observed in \cite{mantelos2026determinantally}, if this is established, it will follow that any two arbitrarily-chosen $3$-cycles in $\Lambda$, $p\coloneqq (p_i)_{i=0}^3$ and $q\coloneqq (q_i)_{i=0}^3$, must belong to the same Case:
	\begin{enumerate}
		\item \textit{If the cycles $p$ and $q$ share the same vertices,} then $p$ and $q$ are obviously the same $3$-cycle up to a possible reversion, in which case $p$ and $q$ belong to the same Case trivially by definition.
		\item \textit{If the cycles $p$ and $q$ share exactly two vertices,} then define $\mathcal{M}$ to be the set of vertices of $p$ and $q$. It is clear that $|\mathcal{M}|=4$ and that $p$ and $q$ are $3$-cycles in $\mathcal{M}$, and so we are done.
		\item \textit{If the cycles $p$ and $q$ share exactly one vertex,} then there is a unique index $i$ such that $q_i$ is a vertex of $p$. Denote by $q_{l_1}$ and $q_{l_2}$ the other two distinct vertices of $q$; and by $p_{m_1}$ and $p_{m_2}$ the other two distinct vertices of $p$. Note how the $3$-cycles $p$ and $(q_i,q_{l_1},p_{m_1},q_i)$ share exactly two vertices; namely, $q_i$ and $p_{m_1}$. By the conclusion of the above scenario, we must then have that the $3$-cycles $p$ and $(q_i,q_{l_1},p_{m_1},q_i)$ belong to the same Case. Furthermore, the $3$-cycles $q$ and $(q_i,q_{l_1},p_{m_1},q_i)$ also share exactly two vertices; namely, $q_i$ and $q_{l_1}$. Hence, for the same reason, it must be the case that the $3$-cycles $q$ and $(q_i,q_{l_1},p_{m_1},q_i)$ belong to the same Case as well. We are done by transitivity.
		\item \textit{If the cycles $p$ and $q$ have completely different vertices,} then since the $3$-cycles $p$ and $(q_0,q_1,p_1,q_0)$ share exactly one vertex, it follows, thanks to the conclusion of the above scenario, that the $3$-cycles $p$ and $(q_0,q_1,p_1,q_0)$ belong to the same Case. Additionally, the $3$-cycles $q$ and $(q_0,q_1,p_1,q_0)$ share exactly two vertices; and so by the conclusion of the second scenario, we then have that the $3$-cycles $q$ and $(q_0,q_1,p_1,q_0)$ belong to the same Case. We are done by transitivity.
	\end{enumerate}

	\section{Further Cycle Identities}
	\label{sec_further_identities}
	
	As noted earlier, many of the identities involving $3$-cycles and $4$-cycles established in \cite{mantelos2026determinantally} are no longer directly applicable in the present setting due to the possibility of vanishing values. In this section, we derive two new identities that are specifically adapted to this more general framework. Although simple in form, these identities prove to be powerful tools later on in our analysis.
	
	Throughout this section $\mathcal{M}\coloneqq\{1,2,3,4\}$, and we abide to the notations for $4$-cycles and $3$-cycles given in \eqref{four_cycle_labellings} and \eqref{three_cycle_labellings}, respectively.
	
	The following lemma supplements Lemma 5.2 of \cite{mantelos2026determinantally} by exhibiting a similar structural connection between $3$-cycles and $4$-cycles, using a slightly weaker hypothesis.
	
	\begin{lemma}
		\label{graph_lemma_1}
		Let $\mathbb{F}$ be a field, and let $h:\mathcal{M}^2\to\mathbb{F}$ be a function such that $h\neq 0$ on $$\{(x,y)\in\mathcal{M}:x\neq y\}\setminus\{(1,3),(3,1),(2,4),(4,2)\}.$$
		\begin{enumerate}[\itshape(i)]
			\item If $h(1,3)=0\text{ and } h(3,1)\neq0,$ then
			\begin{equation*}
				h[q^{[1]}]=h(3,4)h(4,3)\cdot h(4,1)h(1,4)\cdot\frac{h[p^{(1)}]}{h[p^{(3)}]};
			\end{equation*}
			and if instead $h(1,3)\neq0\text{ and } h(3,1)=0,$ then
			\begin{equation*}
				h'[q^{[1]}]=h(3,4)h(4,3)\cdot h(4,1)h(1,4)\cdot\frac{h'[p^{(1)}]}{h'[p^{(3)}]}.
			\end{equation*}
			\item If $h(2,4)=0$ and $h(4,2)\neq 0$, then
			\begin{equation*}
				h[q^{[1]}]=h(4,1)h(1,4)\cdot h(1,2)h(2,1)\cdot \frac{h[p^{(4)}]}{h'[p^{(2)}]};
			\end{equation*}
			and if instead $h(2,4)\neq0$ and $h(4,2)= 0$, then
			\begin{equation*}
				h'[q^{[1]}]=h(4,1)h(1,4)\cdot h(1,2)h(2,1)\cdot \frac{h'[p^{(4)}]}{h[p^{(2)}]}.
			\end{equation*}
		\end{enumerate}	
	\end{lemma}
	
	\begin{proof}
		If 	$h(1,3)=0\text{ and } h(3,1)\neq0,$ then
		\begin{align*}
			h[q^{[1]}]&= h(1,2)h(2,3)h(3,4)h(4,1)\\
			&= \frac{h(1,2)h(2,3)h(3,1)}{h(3,1)}\cdot \frac{h(3,4)h(4,3)}{h(4,3)}\cdot \frac{h(4,1)h(1,4)}{h(1,4)}\\
			&= h(3,4)h(4,3)\cdot h(4,1)h(1,4)\cdot\frac{h[p^{(1)}]}{h[p^{(3)}]}.
		\end{align*}
		One proceeds analogously for the case where $h(1,3)\neq0\text{ and } h(3,1)=0$.
		
		Similarly, if $h(2,4)=0$ and $h(4,2)\neq 0$, then
		\begin{align*}
			h[q^{[1]}]&= h(2,3)h(3,4)h(4,1)h(1,2)\\
			&= \frac{h(2,3)h(3,4)h(4,2)}{h(4,2)}\cdot \frac{h(4,1)h(1,4)}{h(1,4)}\cdot\frac{h(1,2)h(2,1)}{h(2,1)}\\
			&=h(4,1)h(1,4)\cdot h(1,2)h(2,1)\cdot \frac{h[p^{(4)}]}{h'[p^{(2)}]}.
		\end{align*}
		One proceeds analogously for the case where $h(2,4)\neq0$ and $h(4,2)= 0$. \qed
	\end{proof}
	
	The next lemma is of the same flavour as Lemma~5.3 of \cite{mantelos2026determinantally}, but instead of `decomposing' a $3$-cycle in terms of three other distinct $3$-cycles, it does a $4$-cycle in terms of four distinct $3$-cycles.
	
	\begin{lemma}
		\label{graph_lemma_2}
		Let $\mathbb{F}$ be a field, and let $h:\left( \mathcal{M}\cup\{5\}\right)^2\to\mathbb{F}$ be a function of two variables such that for every $x\in\mathcal{M}$, 
		\begin{equation*}
			h(x,5)\neq 0\text{ and }h(5,x)\neq 0.
		\end{equation*}
		%\begin{equation*}
		%	h(1,5)\neq 0, h(5,1)\neq 0, h(2,5)\neq 0, h(5,2)\neq 0, h(3,5)\neq 0, h(5,3)\neq 0, h(4,5)\neq 0, h(5,4)\neq 0. 
		%	\end{equation*}
	Then,
	\begin{equation*}
		h[q^{[1]}]=\frac{h[r^{(1)}]h[r^{(2)}]h[r^{(3)}]h[r^{(4)}]}{h(1,5)h(5,1)\cdot h(2,5)h(5,2)\cdot h(3,5)h(5,3)\cdot h(4,5)h(5,4)},
	\end{equation*}
	where $r^{(1)}$, $r^{(2)}$, $r^{(3)}$, $r^{(4)}$ are the $3$-cycles
	\begin{equation*}
		r^{(1)}\coloneqq (1,2,5,1),\enspace r^{(2)}\coloneqq (2,3,5,2),\enspace r^{(3)}\coloneqq (3,4,5,3),\enspace r^{(4)}\coloneqq (4,1,5,4).
	\end{equation*}
\end{lemma}

\begin{proof}
	In the below figure we provide a pictorial proof of the lemma.
	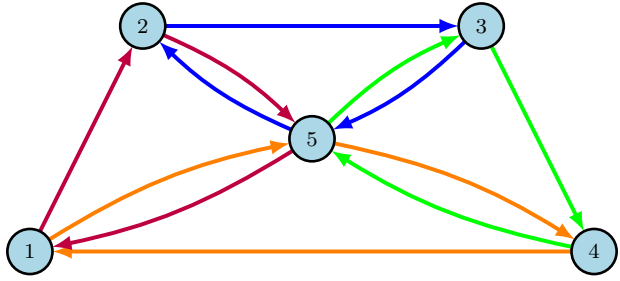
\begin{figure}[H]
		\centering
		\begin{tikzpicture}[scale=1.5]
			\Vertex[label=$1$]{A} \Vertex[x=5,label=$4$]{D} \Vertex[x=1,y=2, label=$2$]{B} \Vertex[x=4,y=2, label=$3$]{C} \Vertex[x=2.5,y=1,label=$5$]{E}
			\Edge[Direct,color=purple](A)(B)
			\Edge[Direct,color=blue](B)(C)
			\Edge[Direct,color=green](C)(D)
			\Edge[Direct,color=orange](D)(A)
			\Edge[Direct,bend=10,color=orange](A)(E)
			\Edge[Direct,bend=10,color=purple](E)(A)
			\Edge[Direct,bend=10,color=purple](B)(E)
			\Edge[Direct,bend=10,color=blue](E)(B)
			\Edge[Direct,bend=10,color=blue](C)(E)
			\Edge[Direct,bend=10,color=green](E)(C)
			\Edge[Direct,bend=10,color=green](D)(E)
			\Edge[Direct,bend=10,color=orange](E)(D)
			%\Edge[Direct,bend=10,color=blue](A)(C)
			%\Edge[Direct,bend=10,color=red](C)(A)	
		\end{tikzpicture}
		\caption{The $4$-cycle $q^{[1]}=(1,2,3,4,1)$ decomposed into four $3$-cycles: in red is the $3$-cycle {\color{purple}$r^{(1)}$}, in blue is the $3$-cycle {\color{blue}$r^{(2)}$}, in green is the $3$-cycle {\color{green}$r^{(3)}$}, and in orange is the $3$-cycle {\color{orange}$r^{(4)}$}. \qed}
	\end{figure}
\end{proof}

\section{Structural Consequences of Property $\mathcal{D}$}
\label{sec_implications_on_D_class_assumption}

This section is devoted to a detailed analysis of the `class $\mathcal{D}$' assumption, which plays a central role in our work. We investigate its structural consequences in the presence of vanishing values and show that it imposes strong constraints on the behaviour of determinantally equivalent functions. Several new results are established that demonstrate how this assumption ensures sufficient regularity for our methods to apply, ultimately enabling the proof of the main theorem.

There are several problematic scenarios arising from (ii) of Lemma~\ref{lemma_det_equivalence_w_easy_consequence} that we need to deal with first before proceeding with our analysis. More specifically, we need to first guarantee, under our hypothesis that the pair of determinantally equivalent functions $K,Q:\Lambda^2\to\mathbb{F}$ are of class $\mathcal{D}$, the non-existence of configurations $(x,y)\in\Lambda^2$ with $x\neq y$ such that 
\begin{equation}
	\label{problematic_scenario_1}
	K(x,y)=0,\enspace K(y,x)=0\text{ and } Q(x,y)=0,\enspace Q(y,x)\neq0
\end{equation}
or
\begin{equation}
	\label{problematic_scenario_2}
	K(x,y)=0,\enspace K(y,x)\neq 0\text{ and } Q(x,y)=0,\enspace Q(y,x)=0
\end{equation}
or
\begin{equation}
	\label{problematic_scenario_3}
	K(x,y)=0,\enspace K(y,x)=0\text{ and } Q(x,y)\neq0,\enspace Q(y,x)\neq0
\end{equation}
or
\begin{equation}
	\label{problematic_scenario_4}
	K(x,y)\neq0,\enspace K(y,x)=0\text{ and } Q(x,y)=0,\enspace Q(y,x)=0.
\end{equation}

The configurations \eqref{problematic_scenario_1}--\eqref{problematic_scenario_4} are precisely problematic because the occurrence of any one of them clearly precludes the existence of functions $S,\tilde{S}: \Lambda^2\to \mathbb{F}$ satisfying \eqref{eq_impossible_1} or \eqref{eq_impossible_2}.

%Indeed, the fact that $Q(y,x)\neq 0$ and $K(y,x)=0$ makes \eqref{eq_impossible_1} impossible, and the fact that $K(x,y)=0$ makes \eqref{eq_impossible_2} impossible. But if neither \eqref{eq_impossible_1} or \eqref{eq_impossible_2} is possible, then by Proposition~4.2 of \cite{mantelos2026determinantally} $Q$ and $K$ cannot possibly be transformed into one another through conjugation and/or transposition transformations!

\begin{remark}
	In \cite{equiv_symm_kernels_for_dpps}, where the pair of determinantally equivalent functions $K,Q:\Lambda^2\to\mathbb{F}$ were taken to be symmetric (i.e., $K(x,y)=K(y,x)$ and $Q(x,y)=Q(y,x)$ for every $x,y\in\Lambda$), the second part of Lemma~\ref{lemma_det_equivalence_w_easy_consequence} ensured
	\begin{equation*}
		K(x,y)=0\iff Q(x,y)=0,\qquad x,y\in\Lambda,
	\end{equation*}
	which automatically rules out all of \eqref{problematic_scenario_1}--\eqref{problematic_scenario_4}.
	
	Likewise, even in the non-symmetric setting of \cite{mantelos2026determinantally}, the additional assumption that $K$ and $Q$ are non-vanishing on the set $\{(x,y)\in\Lambda^2:x\neq y\}$, together with the second part of Lemma~\ref{lemma_det_equivalence_w_easy_consequence}, immediately excludes configurations \eqref{problematic_scenario_1}--\eqref{problematic_scenario_4}.
\end{remark}
In general, for an arbitrary pair of determinantally equivalent functions $K$ and $Q$ there is nothing preventing any of \eqref{problematic_scenario_1}--\eqref{problematic_scenario_4} from occurring. Under the `class $\mathcal{D}$' condition, however, we show that none of these problematic scenarios can ever occur. In particular, we obtain the following result.

\begin{proposition}
	\label{prop_guarantee_nonexistence_of_problematic_pairs}
	Let $\Lambda$ be a set, $\mathbb{F}$ a field, and let $K,Q:\Lambda^2\to\mathbb{F}$ be a pair of determinantally equivalent functions of class $\mathcal{D}$. Then, for every $x,y\in\Lambda$ distinct, either
	\begin{enumerate}[\itshape(i)]
		\item $K(x,y)=K(y,x)=Q(x,y)=Q(y,x)=0$, or 
		\item $K(x,y)\neq 0$, $K(y,x)\neq 0$, $Q(x,y)\neq 0$ and $Q(y,x)\neq 0$, or
		\item $K(x,y)=0=Q(x,y)$ and $K(y,x)\neq 0$, $Q(y,x)\neq 0$, or
		\item $K(x,y)\neq 0$, $Q(x,y)\neq 0$, and $K(y,x)=0=Q(y,x)$, or 
		\item $K(x,y)=0=Q(y,x)$ and $K(y,x)\neq 0$, $Q(x,y)\neq 0$, or
		\item $K(x,y)\neq 0$, $Q(y,x)\neq 0$, and $K(y,x)=0=Q(x,y)$.
	\end{enumerate}
	Specifically, cases (iii) and (iv) pertain to the situation where $(x,y)$ is an edge of a $3$-cycle in $\Lambda$ belonging only to Case 1, and cases (v) and (vi) to where $(x,y)$ is an edge of a $3$-cycle in $\Lambda$ belonging only to Case 2.
\end{proposition}

We start by presenting the first immediate consequence of the `class $\mathcal{D}$' assumption.

%In \cite{mantelos2026determinantally} we saw how important a role $3$-cycles played in establishing the respective result regarding nowhere zero determinantally equivalent functions. While many of the graph theoretic tricks from that paper involving $3$-cycles are not applicable in the current setting with possibly zero valued determinantally equivalent functions, it is again a careful analysis of $3$-cycles that brings about Proposition~\ref{prop_guarantee_nonexistence_of_problematic_pairs}. 

\begin{lemma}
	\label{prop_one_zero_implies_many_nonzeros}
	Let $\Lambda$ be a set, and $\mathbb{F}$ a field. Suppose that the function $h:\Lambda^2\to\mathbb{F}$ is of class $\mathcal{D}$. If there exist $x,y\in\Lambda$ distinct such that $h(x,y)=0$, then for every $z\in\Lambda\setminus\{x,y\}$, $$h(x,z)\neq 0\text{ and }h(z,y)\neq 0.$$
\end{lemma}
\begin{proof}
	Let $x,y,z,w\in\Lambda$ be pairwise distinct. Then, by property $\mathcal{D}$,
	\begin{equation}
		\begin{vmatrix}
			h(x,y) & h(x,z) \\
			h(w,z) & h(w,z)
		\end{vmatrix}=\begin{vmatrix}
			0 & h(x,z) \\
			h(w,z) & h(w,z)
		\end{vmatrix} \neq 0,
	\end{equation}
	in which case we necessarily have $h(x,z)\neq 0$ and $h(w,z)\neq 0$. By also interchanging $z$ and $w$, the result follows. \qed
\end{proof}

\begin{remark}
	Recalling the graphical framework from Section~\ref{sec_notations_graph_id}, Lemma~\ref{prop_one_zero_implies_many_nonzeros} shows that knowledge of an edge of a given $3$-cycle on which $K=0$ provides knowledge of several other edges arising from it on which $K\neq 0$. Consider the $3$-cycle $p=(p_i)_{i=0}^3$ from Section~\ref{sec_notations_graph_id}, for instance, and fix  $p_*\in\Lambda\setminus\{p_0,p_1,p_2\}$:
	\begin{figure}[H]
		\centering
		\begin{tikzpicture}[baseline=(current bounding box.center)]
			\Vertex[label=$p_2$]{A} \Vertex[x=6,label=$p_1$]{B}, \Vertex[x=3,y=3,label=$p_0$]{C} 
			\Edge[Direct,color=black,label=$\neg 0$](A)(C)
			\Edge[Direct,color=black,label=0,color=red](C)(B)
			\Edge[Direct,color=black,label=0,color=blue](B)(A)
		\end{tikzpicture}
		\hfill
		\begin{tikzpicture}[baseline=(current bounding box.center)]
			\Vertex[label=$p_2$]{A} \Vertex[x=6,label=$p_1$]{B}, \Vertex[x=3,y=3,label=$p_0$]{C} 
			\Vertex[x=3,y=1,label=$p_*$]{D}
			\Edge[Direct,color=black,label=$\neg 0$](A)(C)
			\Edge[Direct,color=black,label=0,color=red](C)(B)
			\Edge[Direct,color=black,label=0,color=blue](B)(A)
			\Edge[Direct,color=red!50!blue,label=$\neg 0$,bend=-30](C)(A)
			\Edge[Direct,color=blue,label=$\neg0$,bend=-30](B)(C)
			\Edge[Direct,color=red,label=$\neg0$,bend=-30](A)(B)
			\Edge[Direct, color=red, bend=-30, label=$\neg0$](C)(D)
			\Edge[Direct, color=blue, bend=-10, label=$\neg0$](B)(D)
		\end{tikzpicture}
		\caption{In the left graph, we depict the cycle $p$, with its two known edges on which $K=0$ highlighted in red and blue, respectively. In the right graph, we show the edges on which $K\neq0$ that are implied by Lemma~\ref{prop_one_zero_implies_many_nonzeros}. The coloring reflects the edge from which they arise, with purple indicating edges associated with both.} 
	\end{figure} 
\end{remark}

Lemma~\ref{prop_one_zero_implies_many_nonzeros} brings about two further interesting consequences regarding $3$-cycles, stated as lemmas below.
\begin{lemma}
	\label{lemma_3cycle_only_one_Case}
	Let $\Lambda$ be a set, $\mathbb{F}$ a field, and let $K,Q:\Lambda^2\to\mathbb{F}$ be a pair of determinantally equivalent functions of class $\mathcal{D}$. Fix a $3$-cycle $p=(p_0,p_1,p_2,p_0)$ in $\Lambda$ which belongs only to one Case. Then,
	\begin{enumerate}[\itshape(i)]
		\item $p$ belongs only to Case 1 and $K[p]=0$ $\iff$ $p$ has an edge $(p_{i-1},p_i)$	such that
		\begin{equation*}
			K(p_{i-1},p_i)=0=Q(p_{i-1},p_i),\text{ and }K(p_i,p_{i-1})\neq0,\enspace Q(p_i,p_{i-1})\neq0.
		\end{equation*}
		
		Moreover, if $(x,y)$ is an arbitrary edge of the above $3$-cycle $p$ such that $K(x,y)=0$, then it is necessarily of the same type as $(p_{i-1},p_i)$ in the sense that it also satisfies
		\begin{equation}
			\label{lemma_only_one_case_1_xy_zero}
			K(x,y)=0=Q(x,y),\text{ and }K(y,x)\neq0,\enspace Q(y,x)\neq0.
		\end{equation}
		And if $(x,y)$ instead satisfied $K(x,y)\neq0$, then
		\begin{equation}
			\label{lemma_only_one_case_1_xy_nonzero}
			K(x,y)\neq0,\enspace K(y,x)\neq 0,\enspace Q(x,y)\neq 0,\enspace Q(y,x)\neq 0.
		\end{equation}
		\item $p$ belongs only to Case 2 and $K[p]=0$ $\iff$ $p$ has an edge $(p_{i-1},p_i)$	such that 
		\begin{equation*}
			K(p_{i-1},p_i)=0=Q(p_i,p_{i-1}),\text{ and }K(p_i,p_{i-1})\neq0,\enspace Q(p_{i-1},p_i)\neq 0.
		\end{equation*}
		
		Moreover, if $(x,y)$ is an arbitrary edge of the above $3$-cycle $p$ such that $K(x,y)=0$, then it is necessarily of the same type as $(p_{i-1},p_i)$ in the sense that it also satisfies
		\begin{equation*}
			K(x,y)=0=Q(y,x),\text{ and }K(y,x)\neq0,\enspace Q(x,y)\neq0.
		\end{equation*}
		And if $(x,y)$ instead satisfied $K(x,y)\neq0$, then
		\begin{equation*}
			K(x,y)\neq0,\enspace K(y,x)\neq 0,\enspace Q(x,y)\neq 0,\enspace Q(y,x)\neq 0.
		\end{equation*}
	\end{enumerate}
\end{lemma}
\begin{proof}
	We note that in both instances (i) and (ii) of the lemma $K'[p]\neq 0$. This follows from Lemma~\ref{lemma_3cycle_in_both_cases_condition} and the fact that $K[p]=0$ in both (i) and (ii). From this one realizes that the forward implications of (i) and (ii) are almost immediate -- it is the backward implications that are interesting, and where property $\mathcal{D}$ is extensively used. 
	
	Starting with the proof of the forward direction of instance (i), we note that $K[p]=0$ implies that $K(p_{i-1},p_i)=0$ for at least one index $i$, and since $K'[p]\neq 0$, it must also be that $K(p_{i},p_{i-1})\neq 0$. Since $p$ belongs to Case 1, we have $Q'[p]= K'[p]$ and hence $Q'[p]\neq 0$, in which case $Q(p_{i},p_{i-1})\neq 0$ as well. By the second part of Lemma~\ref{lemma_det_equivalence_w_easy_consequence},
	\begin{equation*}
		0=K(p_{i-1},p_i)K(p_{i},p_{i-1})=Q(p_{i-1},p_i)Q(p_{i},p_{i-1}).
	\end{equation*}
	Since $Q(p_{i},p_{i-1})\neq 0$, it must be that $Q(p_{i-1},p_i)=0$. 
	
	We now prove the more interesting backward direction of instance (i). It follows from $K(p_{i-1},p_i)=0$ that
	\begin{equation*}
		K[p]=K(p_{i-1},p_i)K(p_i,p_{i+1})K(p_{i+1},p_{i-1})=0.
	\end{equation*}
	By the second part of Lemma~\ref{lemma_det_equivalence_w_easy_consequence}, the fact that $K(p_{i-1},p_i)=0$ also brings about
	\begin{equation*}
		0=K(p_{i-1},p_i)K(p_i,p_{i-1})=Q(p_{i-1},p_i)Q(p_i,p_{i-1}).
	\end{equation*}
	Because $Q(p_i,p_{i-1})$ is assumed to be non-zero, it must be that $Q(p_{i-1},p_i)=0$, which of course implies $Q[p]=0$. Therefore,
	\begin{equation*}
		Q[p]=0=K[p],
	\end{equation*}
	and so $p$ indeed belongs to Case 1 and $K[p]=0$. It remains to prove that it does not belong to Case 2, which by Lemma~\ref{lemma_3cycle_in_both_cases_condition} is equivalent to showing that $K'[p]\neq 0$.
	
	Applying Lemma~\ref{prop_one_zero_implies_many_nonzeros} in conjunction with $K(p_{i-1},p_i)=0$ we get
	\begin{equation*}
		K(p_{i-1},p_{i+1})\neq0\text{ and } K(p_{i+1},p_i)\neq 0.
	\end{equation*}
	This, together with the assumed $K(p_{i},p_{i-1})\neq 0$, brings about
	\begin{equation*}
		K'[p]=K(p_{i-1},p_{i+1})K(p_{i+1},p_i)K(p_{i},p_{i-1})\neq 0,
	\end{equation*}
	as required.
	
	Now, if $(x,y)$ is any edge of the $3$-cycle $p$ from (i) of the lemma such that $K(x,y)=0$, then since $K(y,x)\neq 0$ (recall that $K'[p]\neq 0$) and $Q(y,x)\neq 0$ (indeed, if $Q(y,x)$ were zero instead, then $Q'[p]$ would be zero and hence $K[p]=Q'[p]$, contradicting the fact that $p$ only belongs to Case 1), we have by the second part of Lemma~\ref{lemma_det_equivalence_w_easy_consequence} that
	\begin{equation*}
		0=K(x,y)K(y,x)=Q(x,y)Q(y,x),
	\end{equation*}
	in which case $Q(x,y)$ necessarily equals zero. This is precisely \eqref{lemma_only_one_case_1_xy_zero}.
	
	If $(x,y)$ were instead an edge of the $3$-cycle $p$ from (i) of the lemma such that $K(x,y)\neq0$, then this in conjunction with the second part of  Lemma~\ref{lemma_det_equivalence_w_easy_consequence} and the already known $K(y,x)\neq 0$ brings about \eqref{lemma_only_one_case_1_xy_nonzero}.
	
	The proof of part (ii) of the lemma follows analogous steps and so we omit it. \qed
\end{proof}

\begin{lemma}
	\label{lemma_3cycle_both_Cases}
	Let $\Lambda$ be a set, $\mathbb{F}$ a field, and let $K,Q:\Lambda^2\to\mathbb{F}$ be a pair of determinantally equivalent functions of class $\mathcal{D}$. Fix a $3$-cycle $p=(p_0,p_1,p_2,p_0)$ in $\Lambda$ such that $K[p]=0$. Then, $p$ belongs to both Case 1 and Case 2 if and only if there is a unique edge $(p_{j-1},p_j)$ such that
	\begin{equation*}
		K(p_{j-1},p_j)=K(p_j,p_{j-1})=Q(p_{j-1},p_j)=Q(p_j,p_{j-1})=0,
	\end{equation*}
	and the two other edges of $p$, $(p_j,p_{j+1})$ and $(p_{j+1},p_{j-1})$, satisfy
	\begin{equation*}
		K(p_j,p_{j+1})\neq 0,\enspace K(p_{j+1},p_j)\neq 0,\enspace K(p_{j+1},p_{j-1})\neq 0,\enspace K(p_{j-1},p_{j+1})\neq 0,
	\end{equation*}
	and
	\begin{equation*}
		Q(p_j,p_{j+1})\neq 0,\enspace Q(p_{j+1},p_j)\neq 0,\enspace Q(p_{j+1},p_{j-1})\neq 0,\enspace Q(p_{j-1},p_{j+1})\neq 0.
	\end{equation*}
\end{lemma}
\begin{proof}
	By Lemma~\ref{lemma_3cycle_only_one_Case}, the $3$-cycle $p$ from the statement of the current lemma belongs to both Case 1 and Case 2 if and only if there is \textit{no} edge $(p_{i-1},p_i)$ such that $K(p_{i-1},p_i)=0$ and $K(p_i,p_{i-1})\neq 0$, which happens if and only if for every edge $(p_{k-1},p_k)$ of $p$ either
	\begin{equation}
		\label{possibility_1}
		K(p_{k-1},p_k)=K(p_k,p_{k-1})=0
	\end{equation}
	or
	\begin{equation}
		\label{possibility_2}
		K(p_{k-1},p_k)\neq 0\text{ and } K(p_k,p_{k-1})\neq 0
	\end{equation}
	or
	\begin{equation}
		\label{want_show_not_true}
		K(p_{k-1},p_k)\neq 0\text{ and } K(p_k,p_{k-1})= 0.
	\end{equation}
	We first want to show that \eqref{want_show_not_true} never occurs, that is, there is no edge $(p_{k-1},p_k)$ of $p$ that satisfies \eqref{want_show_not_true}. Indeed, suppose for a contradiction that there was an edge $(p_{k-1},p_k)$ such that $K(p_{k-1},p_k)\neq 0$ and $K(p_k,p_{k-1})= 0$. The latter implies, by Lemma~\ref{prop_one_zero_implies_many_nonzeros}, that 
	\begin{equation*}
		K(p_k,p_{k+1})\neq 0.
	\end{equation*}
	But since
	\begin{equation*}
		K[p]=K(p_{k-1},p_k)K(p_k,p_{k+1})K(p_{k+1},p_{k-1})=0,
	\end{equation*}
	it must be that $K(p_{k+1},p_{k-1})=0$, which by Lemma~\ref{prop_one_zero_implies_many_nonzeros} implies
	\begin{equation*}
		K(p_k,p_{k-1})\neq 0,
	\end{equation*}
	a contradiction. Therefore, for an edge $(p_{k-1},p_k)$ of $p$, it is indeed only scenarios \eqref{possibility_1} and \eqref{possibility_2} that are possible. Since $K[p]=0$, there is most certainly at least one edge $(p_{j-1},p_j)$ that satisfies \eqref{possibility_1}, i.e., 
	\begin{equation}
		\label{eq_all_zeroz}
		K(p_{j-1},p_j)=K(p_j,p_{j-1})=0,
	\end{equation}
	in which case, by the second part of Lemma~\ref{lemma_det_equivalence_w_easy_consequence}, at least one of $Q(p_{j-1},p_j)$ or $Q(p_j,p_{j-1})$ must be zero. In fact, both quantities must be zero, that is,
	\begin{equation*}
		Q(p_{j-1},p_j)=Q(p_j,p_{j-1})=0,
	\end{equation*}
	for otherwise, by Lemma~\ref{lemma_3cycle_only_one_Case}, we would have $p$ belonging to only one Case.
	By Lemma~\ref{prop_one_zero_implies_many_nonzeros}, \eqref{eq_all_zeroz} also implies
	\begin{equation*}
		K(p_{j-1},p_{j+1})\neq 0,\enspace K(p_{j+1},p_{j-1})\neq 0,\enspace K(p_{j+1},p_j)\neq 0\text{ and } K(p_j,p_{j+1})\neq 0,
	\end{equation*}
	in which case, by the second part of Lemma~\ref{lemma_det_equivalence_w_easy_consequence},
	\begin{equation*}
		Q(p_{j-1},p_{j+1})\neq 0,\enspace Q(p_{j+1},p_{j-1})\neq 0,\enspace Q(p_{j+1},p_j)\neq 0\text{ and } Q(p_j,p_{j+1})\neq 0,
	\end{equation*}
	also, which proves the uniqueness claim of the lemma. \qed
\end{proof}

\begin{comment}
	\begin{remark}
		\todo{is this necessary?}There is an interesting remark to be made, namely that under the hypothesis that the pair of determinantally equivalent functions $K$ and $Q$ are of class $\mathcal{D}$, there does not exist a $3$-cycle $p=(p_0,p_1,p_2,p_0)$ in $\Lambda$ such that
		\begin{equation*}
			K(p_0,p_1)=K(p_1,p_2)=K(p_2,p_0)=Q(p_0,p_1)=Q(p_1,p_2)=Q(p_2,p_0)=0
		\end{equation*}
		and
		\begin{equation*}
			K(p_1,p_0)=K(p_2,p_1)=K(p_0,p_2)=Q(p_1,p_0)=Q(p_2,p_1)=Q(p_0,p_2)=0.
		\end{equation*}
	\end{remark}
\end{comment}

The two preceding lemmas imply that if two $3$-cycles $p$ and $q$ on the same four vertices belong exclusively to Case 1 and Case 2, respectively, then $K\neq 0$ on their shared edge. More precisely, we have the following corollary.

\begin{corollary}
	\label{cor_impossible_zeroedge_shared_w_Case1_Case2}
	Let $\Lambda$ be a set, $\mathbb{F}$ a field, and let $K,Q:\Lambda^2\to\mathbb{F}$ be a pair of determinantally equivalent functions of class $\mathcal{D}$. Let $p=(p_i)_{i=1}^4$ and $q=(q_i)_{i=1}^4$ be two distinct $3$-cycles in some subset $\mathcal{M}\subseteq\Lambda$ with $|\mathcal{M}|=4$, such that $p$ is not $q$ in reverse. Denote by $$\mathcal{I}\coloneqq\{p_0,p_1,p_2\}\cap\{q_0,q_1,q_2\}$$ the set of common vertices of $p$ and $q$.
	\begin{enumerate}[\itshape(i)]
		\item If $p$ and $q$ belong exclusively to Case 1 and Case 2, respectively, then $$K(x,y)\neq 0\quad\forall x,y\in\mathcal{I}\text{ distinct.}$$
		\item If $p$ and $q$ both belong exclusively to Case 1, or both belong exclusively to Case 2, then either $$K(x,y)\neq 0\quad\forall x,y\in\mathcal{I}\text{ distinct}$$ or $$K(x,y)=0\text{ and }K(y,x)\neq 0\text{ for some $x,y\in\mathcal{I}$ distinct.}$$ 
		\item If $p$ and $q$ each belongs to both Case 1 and Case 2, then either $$K(x,y)\neq 0\quad\forall x,y\in\mathcal{I}\text{ distinct}$$ or $$K(x,y)=0\quad\forall x,y\in\mathcal{I}\text{ distinct}.$$
	\end{enumerate}	
\end{corollary}
\begin{proof}
	For simplicity, set $\mathcal{M}\coloneqq \{1,2,3,4\}$ and set $p=p^{(1)}$ and $q=p^{(2)}$ from \eqref{three_cycle_labellings}. Thus, $$\mathcal{I}=\{1,2\}.$$
	We prove \textit{(i)}: The argument is symmetric, and so we need only consider the case where $K(1,2)=0$. Since $(1,2)$ is an edge of $p^{(1)}$, which we are assuming to belong exclusively to Case 1, the first part of Lemma~\ref{lemma_3cycle_only_one_Case} implies $Q(2,1)\neq 0$. 
	
	However, $(1,2)$ is also an edge of $p^{(2)}$, which we have assumed to belong exclusively to Case 2. Hence, $K(1,2)=0$ implies, by the second part of Lemma~\ref{lemma_3cycle_only_one_Case}, that $Q(2,1)=0$, a contradiction.
	
	We now prove \textit{(ii)}: Suppose that $K(x,y)=0$ for some $x,y\in\mathcal{I}$ distinct. If we also had $K(y,x)=0$, then we would have $$K[p]=K'[p].$$ By Proposition~\ref{lemma_3cycle_in_both_cases_condition}, this implies that $p$ belongs to both Cases, a contradiction.
	
	Claim \textit{(iii)} of the Corollary follows immediately from Lemma~\ref{lemma_3cycle_both_Cases}. \qed
\end{proof}

Using Lemma~\ref{lemma_3cycle_only_one_Case} and Lemma~\ref{lemma_3cycle_both_Cases} we now prove the main result of this section, Proposition~\ref{prop_guarantee_nonexistence_of_problematic_pairs}.

\begin{proof}[Proof of Proposition~\ref{prop_guarantee_nonexistence_of_problematic_pairs}]
	Fix $x,y\in\Lambda$ distinct. There are exactly six possibilities:
	\begin{enumerate}
		\item $(x,y)$ is an edge of a $3$-cycle $p$ in $\Lambda$ such that $K[p]\neq 0$ and $K'[p]\neq 0$. 
		\item $(x,y)$ is an edge of a $3$-cycle $p$ in $\Lambda$ belonging to both Case 1 and Case 2 such that $K[p]=0$. 
		\item $(x,y)$ is an edge of a $3$-cycle $p$ in $\Lambda$ belonging only to Case 1 such that $K[p]=0$.
		\item $(x,y)$ is an edge of a $3$-cycle $p$ in $\Lambda$ belonging only to Case 1 such that $K'[p]=0$.
		\item $(x,y)$ is an edge of a $3$-cycle $p$ in $\Lambda$ belonging only to Case 2 such that $K[p]=0$.
		\item $(x,y)$ is an edge of a $3$-cycle $p$ in $\Lambda$ belonging only to Case 2 such that $K'[p]=0$. 
	\end{enumerate}
	In the first scenario, $K(x,y)\neq 0$ and $K(y,x)\neq 0$ necessarily, which by Lemma~\ref{lemma_det_equivalence_w_easy_consequence} implies that $Q(x,y)\neq 0$ and $Q(y,x)\neq 0$ as well. Therefore, $(x,y)$ pertains to (i) of Proposition~\ref{prop_guarantee_nonexistence_of_problematic_pairs}.
	
	In the second scenario, since $p$ belongs to both Case 1 and Case 2, Lemma~\ref{lemma_3cycle_both_Cases} immediately tells us that $(x,y)$ pertains to either (i) or (ii) of Proposition~\ref{prop_guarantee_nonexistence_of_problematic_pairs}.
	
	Part (i) of Lemma~\ref{lemma_3cycle_only_one_Case} ensures that in the third scenario, $(x,y)$ pertains to either (iii) or (ii) of Proposition~\ref{prop_guarantee_nonexistence_of_problematic_pairs}; and in the fourth, that $(x,y)$ pertains to either (iv) or (ii) of Proposition~\ref{prop_guarantee_nonexistence_of_problematic_pairs}.
	
	In the same way, part (ii) of Lemma~\ref{lemma_3cycle_only_one_Case} ensures that in the fifth scenario, $(x,y)$ pertains to either (v) or (ii) of Proposition~\ref{prop_guarantee_nonexistence_of_problematic_pairs}; and in the sixth, that $(x,y)$ pertains to either (vi) or (ii) of Proposition~\ref{prop_guarantee_nonexistence_of_problematic_pairs}. 
	
	This exhausts all possibilities for $(x,y)\in\Lambda^2$ and so we are done. \qed
\end{proof}

\begin{remark}
	This result ensures for a pair of determinantally equivalent functions $K,Q:\Lambda^2\to\mathbb{F}$ of class $\mathcal{D}$ within the `all $3$-cycles in $\Lambda$ belong to Case 1' framework that for all $x,y\in\Lambda$,
	\begin{equation*}
		K(x,y)=0 \iff Q(x,y)=0.
	\end{equation*}
	It also ensures that within the `all $3$-cycles in $\Lambda$ belong to Case 2' framework, for all $x,y\in\Lambda$,
	\begin{equation*}
		K(y,x)=0 \iff Q(x,y)=0.
	\end{equation*}
	In the remainder of the paper, these two facts are going to be used implicitly and extensively, without any further explanation. For example, if we were working in the `all $3$-cycles in $\Lambda$ belong to Case 1' framework, and we were to assume that $K(x,y)=0$, then we are also implicitly assuming $Q(x,y)=0$ without any further reasoning or justification (provided $K$ and $Q$ is a pair of determinantally equivalent functions of class $\mathcal{D}$, of course).
\end{remark}

The following identity involving class $\mathcal{D}$ determinantally equivalent functions, stated as a lemma below, will prove useful in establishing the well-definedness of the functions $S$ and $\tilde{S}$ from \eqref{eq:def_S} and \eqref{eq:def_S_tilde}, respectively.

\begin{lemma}
	\label{lemma_consistency_of_expressions}
	Let $\Lambda$ be a set, $\mathbb{F}$ a field, and let $K,Q:\Lambda^2\to\mathbb{F}$ be a pair of determinantally equivalent functions of class $\mathcal{D}$. Fix $x,y\in\Lambda$ distinct.
	
	\begin{enumerate}[\itshape(i)]
		\item If $K(x,y)=0=K(y,x)$ and every $3$-cycle in $\Lambda$ belongs to Case 1, then
		\begin{equation*}
			\frac{Q(x,z)Q(z,y)}{K(x,z)K(z,y)}=\frac{Q(x,w)Q(w,y)}{K(x,w)K(w,y)},\qquad z,w\in\Lambda\setminus\{x,y\}\text{ distinct.}
		\end{equation*}
		If $K(x,y)=0$ and $K(y,x)\neq 0$, then for every $z,w\in\Lambda\setminus\{x,y\}$ distinct,
		\begin{equation*}
			\frac{Q(x,z)Q(z,y)}{K(x,z)K(z,y)}=\left( \frac{Q(y,x)}{K(y,x)} \right)^{-1}=\frac{Q(x,w)Q(w,y)}{K(x,w)K(w,y)}.
		\end{equation*}
		\item If $K(x,y)=0=K(y,x)$ and every $3$-cycle in $\Lambda$ belongs to Case 2, then 
		\begin{equation*}
			\frac{Q(x,z)Q(z,y)}{K(z,x)K(y,z)}=\frac{Q(x,w)Q(w,y)}{K(w,x)K(y,w)},\qquad z,w\in\Lambda\setminus\{x,y\}\text{ distinct.}
		\end{equation*}
		If $K(y,x)=0$ and $K(x,y)\neq 0$, then for every $z,w\in\Lambda\setminus\{x,y\}$ distinct,
		\begin{equation*}
			\frac{Q(x,z)Q(z,y)}{K(z,x)K(y,z)}=\left( \frac{Q(y,x)}{K(x,y)} \right)^{-1}=\frac{Q(x,w)Q(w,y)}{K(w,x)K(y,w)}.
		\end{equation*}
	\end{enumerate}
\end{lemma}
\begin{proof}
	\begin{comment}
		To prove the first part, simply note 
		\begin{align*}
			\frac{Q(x,z)Q(z,y)}{K(x,z)K(z,y)}=\frac{Q(x,z)Q(z,y)Q(y,x)}{K(x,z)K(z,y)K(y,x)}\cdot \left( \frac{Q(y,x)}{K(y,x)} \right)^{-1}&=\left( \frac{Q(y,x)}{K(y,x)} \right)^{-1}\\
			&=\frac{Q(x,w)Q(w,y)Q(y,x)}{K(x,w)K(w,y)K(y,x)}\cdot\left( \frac{Q(y,x)}{K(y,x)} \right)^{-1}\\
			&=\frac{Q(x,w)Q(w,y)}{K(x,w)K(w,y)},
		\end{align*}
		using the fact that $(x,z,y,x)$ and $(x,w,y,x)$ both belong to Case 1. One proves the second part analogously.
	\end{comment}
	For the first part, we begin by making the following observation: since $K(x,y)=0$, Lemma~\ref{prop_one_zero_implies_many_nonzeros} ensures
	\begin{equation}
		\label{eqn_allnonzero_implications}
		K(x,z)\neq 0,\enspace K(z,y)\neq 0,\enspace K(x,w)\neq 0,\enspace K(w,y)\neq 0,
	\end{equation}
	which, in turn, ensures that all the quantities appearing in the first part of the lemma are well-defined.
	
	Next, since we also have $K(y,x)=0$, Lemma~\ref{prop_one_zero_implies_many_nonzeros} further ensures
	\begin{equation*}
		K(y,z)\neq 0,\enspace K(z,x)\neq 0, \enspace K(y,w)\neq 0,\enspace K(w,x)\neq 0.
	\end{equation*}
	
	We can thus write
	
	\begin{align*}
		\frac{Q(x,z)Q(z,y)}{K(x,z)K(z,y)}&=\frac{Q(x,z)Q(z,y)Q(y,w)Q(w,x)}{K(x,z)K(z,y)K(y,w)K(w,x)}\cdot \left( \frac{Q(y,w)Q(w,x)}{K(y,w)K(w,x)} \right)^{-1}\\
		&=\frac{Q(x,z)Q(z,y)Q(y,w)Q(w,x)}{K(x,z)K(z,y)K(y,w)K(w,x)}\cdot\frac{Q(x,w)Q(w,y)}{K(x,w)K(w,y)} &&\text{(by Lemma~\ref{lemma_det_equivalence_w_easy_consequence})}.
	\end{align*}
	To simplify notation, we henceforth make use of the following identification:
	\begin{equation*}
		x=1,\enspace y=2,\enspace z=3,\enspace w=4,
	\end{equation*}
	in which case, by recalling \eqref{four_cycle_labellings} and \eqref{three_cycle_labellings}, the previous equation becomes
	\begin{equation}
		\label{eqn_immediate_follow}
		\frac{Q(1,3)Q(3,2)}{K(1,3)K(3,2)}=\frac{Q[q^{[3]}]}{K[q^{[3]}]}\cdot\frac{Q(1,4)Q(4,2)}{K(1,4)K(4,2)}.
	\end{equation}
	Now, if $K(3,4)\neq 0$ and $K(4,3)\neq 0$, then we can invoke Lemma~5.2 from \cite{mantelos2026determinantally} (with a simple relabelling of vertices) to get
	\begin{equation*}
		K[q^{[3]}]=\frac{K'[p^{(3)}]K'[p^{(4)}]}{K(3,4)K(4,3)}=\frac{Q'[p^{(3)}]Q'[p^{(4)}]}{Q(3,4)Q(4,3)}=Q[q^{[3]}],
	\end{equation*}
	by Lemma~\ref{lemma_det_equivalence_w_easy_consequence} and the fact that all $3$-cycles are assumed to belong to Case 1. Thus,
	\begin{equation}
		\label{eqn_wts_now}
		\frac{Q(x,z)Q(z,y)}{K(x,z)K(z,y)}\left( =\frac{Q(1,3)Q(3,2)}{K(1,3)K(3,2)}=\frac{Q(1,4)Q(4,2)}{K(1,4)K(4,2)}=\right)\frac{Q(x,w)Q(w,y)}{K(x,w)K(w,y)},
	\end{equation}
	as required.
	
	Let us now suppose we instead had $K(3,4)= 0$ and $K(4,3)\neq 0$. Using the fact that the $3$-cycles $p^{(4)}$ and $p^{(3)}$ both belong to Case 1, we obtain through an application of Lemma~\ref{graph_lemma_1} (with a simple relabelling of vertices)
	\begin{align*}
		K[q^{[3]}]&=K(1,3)K(3,1)\cdot K(4,1)K(1,4)\cdot \frac{K'[p^{(4)}]}{K[p^{(3)}]}\\
		&=Q(1,3)Q(3,1)\cdot Q(4,1)Q(1,4)\cdot \frac{Q'[p^{(4)}]}{Q[p^{(3)}]} &&\text{(by Lemma~\ref{lemma_det_equivalence_w_easy_consequence})}\\
		&=Q[q^{[3]}].
	\end{align*}
	Equation \eqref{eqn_wts_now} now immediately follows from \eqref{eqn_immediate_follow}.
	
	\begin{comment}
		Since we are assuming $K(1,2)=K(2,1)=0=Q(2,1)=Q(1,2)$, in which case
		\begin{equation}
			\label{eq_always_true}
			K[q^{[1]}]=K'[q^{[1]}]=0=Q[q^{[1]}]=Q'[q^{[1]}]
		\end{equation}
		and
		\begin{equation*}
			K[q^{[2]}]=K'[q^{[2]}]=0=Q[q^{[2]}]=Q'[q^{[2]}],
		\end{equation*}
		the first part of Lemma 6.6 in \cite{mantelos2026determinantally} yields
		\begin{equation*}
			K[q^{[3]}]+K'[q^{[3]}]=Q[q^{[3]}]+Q'[q^{[3]}].
		\end{equation*}
		Since, by Lemma~\ref{lemma_det_equivalence_w_easy_consequence}, we also have
		\begin{equation*}
			K[q^{[3]}]\cdot K'[q^{[3]}]=Q[q^{[3]}]\cdot Q'[q^{[3]}],
		\end{equation*}
		it follows that
		\begin{equation*}
			K[q^{[3]}]=Q[q^{[3]}]\text{ and } K'[q^{[3]}]=Q'[q^{[3]}],\text{ or, } K[q^{[3]}]=Q'[q^{[3]}]\text{ and } K'[q^{[3]}]=Q[q^{[3]}].
		\end{equation*}
		If the former is true, then \eqref{eqn_wts_now} holds immediately, and we are done. Suppose now the latter is true. 
	\end{comment}
	
	Let us now suppose that $K(3,4)\neq 0$ and $K(4,3)= 0$. Applying Lemma~\ref{graph_lemma_1} again in the same fashion, we get
	\begin{equation*}
		K'[q^{[3]}]=Q'[q^{[3]}].
	\end{equation*}
	Since, by Lemma~\ref{lemma_det_equivalence_w_easy_consequence},
	\begin{equation*}
		K[q^{[3]}]K'[q^{[3]}]=Q[q^{[3]}]Q'[q^{[3]}],
	\end{equation*}
	it follows that $K[q^{[3]}]=Q[q^{[3]}]$. Equation \eqref{eqn_wts_now} once again immediately follows from \eqref{eqn_immediate_follow}.
	
	Finally, let us suppose that $K(3,4)= 0$ and $K(4,3)= 0$. Fix an element $\xi\in\Lambda\setminus\{x,y,z,w\}$. Lemma~\ref{prop_one_zero_implies_many_nonzeros} then ensures that
	\begin{equation*}
		K(3,\xi)\neq 0,\enspace K(\xi,3)\neq 0,\enspace K(4,\xi)\neq 0,\enspace K(\xi,4)\neq 0.
	\end{equation*}
	Since we also have $K(1,2)=0=K(2,1)$, it also follows, again by Lemma~\ref{prop_one_zero_implies_many_nonzeros}, that
	\begin{equation*}
		K(2,\xi)\neq 0,\enspace K(\xi,2)\neq 0,\enspace K(1,\xi)\neq 0,\enspace K(\xi,1)\neq 0.
	\end{equation*}
	With these non-vanishing terms in our disposal, and using the fact that all $3$-cycles in $\Lambda$ belong to Case 1, we can apply Lemma~\ref{graph_lemma_2} (with a simple relabelling of vertices) in conjunction with the second part of Lemma~\ref{lemma_det_equivalence_w_easy_consequence} to get
	\begin{equation*}
		K[q^{[3]}]=Q[q^{[3]}].
	\end{equation*}
	
	Once again, equation \eqref{eqn_wts_now} follows from \eqref{eqn_immediate_follow}.
	
	Finally, suppose we had $K(x,y)=0$ (which we saw earlier implies \eqref{eqn_allnonzero_implications}) and $K(y,x)\neq 0$, then
	\begin{equation*}
		\frac{Q(x,z)Q(z,y)Q(y,x)}{K(x,z)K(z,y)K(y,x)}=1 
	\end{equation*}
	and
	\begin{equation*}
		\frac{Q(x,w)Q(w,y)Q(y,x)}{K(x,w)K(w,y)K(y,x)}=1,
	\end{equation*}
	since every $3$-cycle in $\Lambda$ is assumed to belong to Case 1. The required equations of the lemma immediately follow.
	
	Analogous steps can be taken to prove the second part of the lemma. \qed
\end{proof}

\begin{remark}
	Let $p$ and $q$ be two distinct $3$-cycles in $\Lambda$, each belonging to both Case 1 and Case 2 and satisfying $K[p]=0=K[q]$. By Lemma~\ref{lemma_3cycle_both_Cases}, each has a unique edge on which $K=0$. Suppose that these two unique edges coincide, and denote this common edge by $(x,y)\in\Lambda^2$. Then necessarily $$p=(z,x,y,z)\text{ and }q=(w,x,y,w),$$ for some $z,w\in\Lambda\setminus\{x,y\}$ distinct. 
	
	It then follows from Lemma~\ref{lemma_consistency_of_expressions} that, in the `all $3$-cycles belong to Case 1' framework,
	\begin{equation*}
		\frac{Q(x,z)Q(z,y)}{K(x,z)K(z,y)}=\frac{Q(x,w)Q(w,y)}{K(x,w)K(w,y)};
	\end{equation*}
	and 
	\begin{equation*}
		\frac{Q(x,z)Q(z,y)}{K(z,x)K(y,z)}=\frac{Q(x,w)Q(w,y)}{K(w,x)K(y,w)},
	\end{equation*}
	in the `all $3$-cycles belong to Case 2' framework.
\end{remark}

With the above structural results at hand, we are now in the position to discuss the well-definedeness of the functions $S$ and $\tilde{S}$ from \eqref{eq:def_S} and \eqref{eq:def_S_tilde}, respectively.

\begin{lemma} 
	\label{lemma_welldefined}
	Let $\Lambda$ be a set, $\mathbb{F}$ a field, and let $K,Q:\Lambda^2\to\mathbb{F}$ be a pair of determinantally equivalent functions of class $\mathcal{D}$. If all $3$-cycles in $\Lambda$ belong to Case 1 (resp. Case 2), then the function $S$ from \eqref{eq:def_S} (resp. the function $\tilde{S}$ from \eqref{eq:def_S_tilde}) is well-defined.
\end{lemma}
\begin{proof}
	We first show that for every edge $(x,y)\in\Lambda^2$ of a $3$-cycle in $\Lambda$ belonging to both Case 1 and Case 2 such that $K(x,y)=0$, the expression
	\begin{equation}
		\label{welldefined_z_in_def_of_S}
		\frac{Q(x,z)Q(z,y)}{K(x,z)K(z,y)},
	\end{equation}
	appearing in the definition of $S$, is well-defined for every $z\in\Lambda\setminus\{x,y\}$; and that for every edge $(x,y)\in\Lambda^2$ of a $3$-cycle in $\Lambda$ belonging to both Case 1 and Case 2 such that $K(y,x)=0$, the expression
	\begin{equation}
		\label{welldefined_z_in_def_of_tildeS}
		\frac{Q(x,z)Q(z,y)}{K(z,x)K(y,z)},
	\end{equation}
	appearing in the definition of $\tilde{S}$, is well-defined for every $z\in\Lambda\setminus\{x,y\}$.

	%Lemma~\ref{??} establishes that the choice of element $z\in\Lambda\setminus\{x,y\}$ appearing in the definition of both $S$ and $\tilde{S}$ is not important; in other words, it establishes that for every zero-edge $(x,y)$ of a $3$-cycle in $\Lambda$ (in the sense that $K(x,y)=0$) belonging to both Case 1 and Case 2, 
	%\begin{equation}
	%\frac{Q(x,z)Q(z,y)}{K(x,z)K(z,y)},
	%\end{equation}
	%for all $w\in\Lambda\setminus\{x,y\}$ not equal to $z$ (from the definition of $S$);
	%and that for every zero-edge $(x,y)$ of a $3$-cycle in $\Lambda$ (in the sense that $K(y,x)=0$) belonging to both Case 1 and Case 2,
	%\begin{equation}
	%\frac{Q(x,z)Q(z,y)}{K(z,x)K(y,z)}=\frac{Q(x,w)Q(w,y)}{K(w,x)K(y,w)},
	%\end{equation}
	%for all $w\in\Lambda\setminus\{x,y\}$ not equal to $z$ (from the definition of $\tilde{S}$).

	To that end, first notice how in \eqref{welldefined_z_in_def_of_S} the denominator $K(x,z)K(z,y)$ is non-zero: $K(x,y)=0$ implies, by Lemma~\ref{prop_one_zero_implies_many_nonzeros}, that $K(x,z)\neq 0$ and $K(z,y)\neq 0$. In the same way, $K(y,x)=0$, in conjunction with Lemma~\ref{prop_one_zero_implies_many_nonzeros}, ensures that the denominator appearing in \eqref{welldefined_z_in_def_of_tildeS} is also non-zero. 
	
	We also need to show that if $(x,y)\in \Lambda^2$ is an edge of a $3$-cycle belonging only to Case 1 such that $K(x,y)=0$, then the expression
	\begin{equation*}
		\left( \frac{Q(y,x)}{K(y,x)} \right)^{-1},
	\end{equation*}
	appearing in the definition of $S$, is well-defined. Indeed, it is: the first part of Lemma~\ref{lemma_3cycle_only_one_Case} ensures that $Q(y,x)\neq 0$ and $K(y,x)\neq 0$.
	
	In the same way, the second part of Lemma~\ref{lemma_3cycle_only_one_Case} ensures that for every edge $(x,y)\in \Lambda^2$ of a $3$-cycle belonging only to Case 2 such that $K(y,x)=0$, the expression
	\begin{equation*}
		\left( \frac{Q(y,x)}{K(x,y)} \right)^{-1},
	\end{equation*}
	appearing in the definition of $\tilde{S}$, is well-defined.
	
	The fourth clause in the definition of $S$ (resp. $\tilde{S}$), under the framework where all $3$-cycles belong to Case 1 (resp. Case 2), was actually already shown to be well-defined in the remark just after the proof of Lemma~\ref{lemma_consistency_of_expressions}.
	
	Finally, no $3$-cycle belonging to both Case 1 and Case 2 can share an edge on which $K=0$ with a $3$-cycle belonging to only one Case. Indeed, by Lemma~\ref{lemma_3cycle_both_Cases}, a $3$-cycle of the former type has a unique such edge, say $(\alpha,\beta)\in\Lambda^2$, satisfying $$K(\alpha,\beta)=K(\beta,\alpha)=0=Q(\alpha,\beta)=Q(\beta,\alpha).$$ Suppose that a $3$-cycle $r$ belonging to only one Case also contains this edge. Then $$K(\alpha,\beta)=K(\beta,\alpha)=0$$ implies that $K[r]=K'[r]$, and hence, by Lemma~\ref{lemma_3cycle_in_both_cases_condition}, the $3-$cycle $r$ must belong to both Case 1 and Case 2, a contradiction. Consequently, no such overlap can occur, and therefore the clauses in the definitions of $S$ and $\tilde{S}$ are mutually exclusive. \qed
\end{proof}

\section{Proof of Theorem~\ref{main_thm}}
\label{sec_proof_of_main_result}

In this final section, we bring together the tools and results developed in the preceding sections to establish the main theorem. The proof synthesizes the combinatorial framework, the new identities, and the structural properties derived from the `class $\mathcal{D}$' assumption, completing the argument.
\begin{proposition}
	\label{prop_cocycle_property_for_S_tildeS}
	Let $\Lambda$ be a set, $\mathbb{F}$ a field, and let $K,Q:\Lambda^2\to\mathbb{F}$ be a pair of determinantally equivalent functions of class $\mathcal{D}$.
	\begin{enumerate}[\itshape(i)]
		\item If all $3$-cycles in $\Lambda$ belong to Case 1, then the function $S$ from \eqref{eq:def_S} is a cocycle function such that for every $x,y\in\Lambda$,
		\begin{equation}
			\label{S_is_appropriate}
			Q(x,y)=S(x,y)K(x,y).
		\end{equation}	
		\item If all $3$-cycles in $\Lambda$ belong to Case 2, then the function $\tilde{S}$ from \eqref{eq:def_S_tilde} is a cocycle function such that for every $x,y\in\Lambda$,
		\begin{equation}
			\label{tildeS_is_appropriate}
			Q(x,y)=\tilde{S}(x,y)K(y,x).
		\end{equation}
	\end{enumerate}	
\end{proposition}
\begin{proof}
	We only prove part (i) of the proposition, since the proof of part (ii) follows analogous steps.
	
	The well-definedness of the function $S$ was already addressed in Lemma~\ref{lemma_welldefined}, and so we proceed directly with the proof of \eqref{S_is_appropriate}. Fix $x,y\in\Lambda$ distinct, then
	\begin{equation*}
		Q(x,x)=S(x,x)K(x,x)=K(x,x)
	\end{equation*}
	follows from the first part of Lemma~\ref{lemma_det_equivalence_w_easy_consequence}. If $K(x,y)\neq 0$, then
	\begin{equation*}
		S(x,y)K(x,y)=\frac{Q(x,y)}{K(x,y)}K(x,y)=Q(x,y);
	\end{equation*}
	and if $K(x,y)=0$, then $Q(x,y)=0$, in which case
	\begin{equation*}
		Q(x,y)=S(x,y)K(x,y)
	\end{equation*}
	trivially holds. Equation \eqref{S_is_appropriate} is now confirmed.

	We now proceed with the proof of the cocycle property. Let $x,y\in\Lambda$ distinct, then
	\begin{equation*}
		S(x,x)=1
	\end{equation*}
	comes directly from the definition of $S$. 
	
	If $(x,y)$ pertains to configuration (i) of Proposition~\ref{prop_guarantee_nonexistence_of_problematic_pairs}, then we know from the proof that it pertains to the situation where $K(x,y)=0$ and $(x,y)$ is an edge of a $3$-cycle $p$ in $\Lambda$ belonging to both Case 1 and Case 2. As such, $p=(z,x,y,z)$, for some $z\in\Lambda\setminus\{x,y\}$. Since $K(y,x)=0$ and $(y,x)$ is an edge of $p'$ (which also belongs to both Case 1 and Case 2), then by the definition of $S$,
	\begin{equation*}
		S(x,y)S(y,x)=\frac{Q(x,z)Q(z,y)}{K(x,z)K(z,y)}\cdot\frac{Q(y,z)Q(z,x)}{K(y,z)K(z,x)}=1,
	\end{equation*}
	where the last equality follows from the second part of Lemma~\ref{lemma_det_equivalence_w_easy_consequence}.
	
	If $(x,y)$ pertains to configuration (ii) of Proposition~\ref{prop_guarantee_nonexistence_of_problematic_pairs}, then Lemma~\ref{lemma_det_equivalence_w_easy_consequence} again ensures that 
	\begin{equation*}
		S(x,y)S(y,x)=\frac{Q(x,y)}{K(x,y)}\frac{Q(y,x)}{K(y,x)}=1.
	\end{equation*}
	
	If $(x,y)$ pertains to configuration (iii) of Proposition~\ref{prop_guarantee_nonexistence_of_problematic_pairs}, which we know from the proof pertains to the case where $K(x,y)=0$ and $(x,y)$ is an edge of a $3$-cycle in $\Lambda$ belonging only to Case 1, then by the definition of $S$, 
	\begin{equation*}
		S(x,y)=\left( \frac{Q(y,x)}{K(y,x)} \right)^{-1}\frac{Q(y,x)}{K(y,x)}=1.
	\end{equation*}
	
	If $(x,y)$ pertains to configuration (iv) of Proposition~\ref{prop_guarantee_nonexistence_of_problematic_pairs}, which we know from the proof pertains to the case where $K(y,x)=0$ and $(y,x)$ is an edge of a $3$-cycle in $\Lambda$ belonging only to Case 1, we have by the definition of $S$,
	\begin{equation*}
		S(x,y)S(y,x)=\frac{Q(x,y)}{K(x,y)}\left( \frac{Q(x,y)}{K(x,y)} \right)^{-1}=1.
	\end{equation*}
	
	Now that we have established the cocycle property of $S$ for $1$- and $2$-cycles, it remains to prove it for $3$-cycles; we proceed by cases:
	\begin{enumerate}
		\item \textit{$p=(p_i)_{i=0}^3$ is a $3$-cycle in $\Lambda$ such that $K[p]\neq 0$}.
		
		In this case, since $p$ belongs to Case 1,
		\begin{equation*}
			S[p]=\frac{Q[p]}{K[p]}=1.
		\end{equation*}
		
		\item \textit{$p=(p_i)_{i=0}^3$ is a $3$-cycle in $\Lambda$ belonging to both Case 1 and Case 2 such that $K[p]=0$}.
		
		Abiding to the notations of Lemma~\ref{lemma_3cycle_both_Cases}, we have, by the definition of $S$,
		\begin{align*}
			S[p]&=S(p_{j-1},p_j)S(p_j,p_{j+1})S(p_{j+1},p_{j-1}) \\
			&= \frac{Q(p_{j-1},p_{j+1})Q(p_{j+1},p_j)}{K(p_{j-1},p_{j+1})K(p_{j+1},p_j)}\cdot  \frac{Q(p_{j},p_{j+1})}{K(p_{j},p_{j+1})}\cdot\frac{Q(p_{j+1},p_{j-1})}{K(p_{j+1},p_{j-1})}\\
			=1,
		\end{align*}
		where the last equality follows from the second part of Lemma~\ref{lemma_det_equivalence_w_easy_consequence}.
		
		\item \textit{$p=(p_i)_{i=0}^3$ is a $3$-cycle in $\Lambda$ belonging only to Case 1 such that $K[p]=0$, with one edge on which $K=0$ and two edges on which $K\neq 0$}.
		
		This is precisely the setting of the first part of Lemma~\ref{lemma_3cycle_only_one_Case}. Consistent with the notation therein, let $(p_{i-1},p_i)$ be the edge of $p$ on which $K=0$, and $(p_i,p_{i+1})$, $(p_{i+1},p_{i-1})$ the edges on which $K\neq 0$. As such, the aforementioned lemma ensures additionally that $K(p_i,p_{i-1})\neq 0$; and so the first part of Lemma~\ref{lemma_consistency_of_expressions} yields 
		\begin{equation}
			\label{eq_nice_help}
			\left( \frac{Q(p_i,p_{i-1})}{K(p_i,p_{i-1})} \right)^{-1}=\frac{Q(p_{i-1},p_{i+1})Q(p_{i+1},p_i)}{K(p_{i-1},p_{i+1})K(p_{i+1},p_i)}.
		\end{equation}
		By the definition of $S$,
		\begin{equation*}
			S[p]=S(p_{i-1},p_i)S(p_i,p_{i+1})S(p_{i+1},p_{i-1})= \left( \frac{Q(p_i,p_{i-1})}{K(p_i,p_{i-1})} \right)^{-1}\cdot \frac{Q(p_i,p_{i+1})}{K(p_i,p_{i+1})}\cdot \frac{Q(p_{i+1},p_{i-1})}{K(p_{i+1},p_{i-1})}=1,
		\end{equation*}
		where the last equality follows from \eqref{eq_nice_help} in conjunction with Lemma~\ref{lemma_det_equivalence_w_easy_consequence}.
		
		\item \textit{$p=(p_i)_{i=0}^3$ is a $3$-cycle in $\Lambda$ belonging only to Case 1 such that $K[p]=0$, with two edges on which $K=0$ and one edge on which $K\neq0$}.
		
		Suppose that $(p_0,p_1)$ and $(p_1,p_2)$ are the two edges of $p$ on which $K=0$ and $(p_2,p_0)$ the edge on which $K\neq0$. As such, the first part of Lemma~\ref{lemma_3cycle_only_one_Case} ensures that $K(p_1,p_0)\neq 0$; and so we can then invoke the first part of Lemma~\ref{lemma_consistency_of_expressions} to get
		\begin{equation}
			\label{eq_nice_help2}
			\left( \frac{Q(p_1,p_0)}{K(p_1,p_0)} \right)^{-1}=\frac{Q(p_0,p_2)Q(p_2,p_1)}{K(p_0,p_2)K(p_2,p_1)}.
		\end{equation}
		By the definition of $S$,
		\begin{equation*}
			S[p]=S(p_{0},p_1)S(p_1,p_{2})S(p_{2},p_{0})=\left( \frac{Q(p_1,p_0)}{K(p_1,p_0)} \right)^{-1}\cdot \left( \frac{Q(p_2,p_1)}{K(p_2,p_1)} \right)^{-1}\cdot \frac{Q(p_2,p_0)}{K(p_2,p_0)}=1,
		\end{equation*}
		where the last equality follows from \eqref{eq_nice_help2} in conjunction with Lemma~\ref{lemma_det_equivalence_w_easy_consequence} (ii).
		
		\item \textit{$p=(p_i)_{i=0}^3$ is a $3$-cycle in $\Lambda$ belonging only to Case 1 such that $K[p]=0$, with $K=0$ on all three of its edges}.
		
		In this case, by the very definition of $S$,
		\begin{align*}
			S[p]=S(p_{0},p_1)S(p_1,p_{2})S(p_{2},p_{0})&=\left( \frac{Q(p_1,p_0)}{K(p_1,p_0)} \right)^{-1}\cdot \left( \frac{Q(p_2,p_1)}{K(p_2,p_1)} \right)^{-1}\cdot \left( \frac{Q(p_0,p_2)}{K(p_0,p_2)} \right)^{-1}\\
			&=\left( \frac{Q'[p]}{K'[p]} \right)^{-1}\\
			&=1,
		\end{align*}
		where the last equality follows from the fact that $p$ belongs to Case 1.
	\end{enumerate}
	Since we have proven that $S[p]=1$ for every possible type of $3$-cycle $p$ in $\Lambda$, it follows that $S$ is a (full) cocycle function. \qed
\end{proof}

With this proposition at hand, to prove Theorem~\ref{main_thm} it remains to show that all $3$-cycles in $\Lambda$ belong to Case 1, or all to Case 2. As was remarked at the end of Section~\ref{sec_strategy}, instead of working with the entire set $\Lambda$, it suffices to prove this for subsets $\mathcal{M}\subseteq \Lambda$ such that $|\mathcal{M}|=4$. More specifically, it suffices to prove the following proposition.

\begin{proposition}
	\label{prop_step_2_lambda_4_elements}
	Let $\Lambda$ be a set, $\mathbb{F}$ a field, and let $K,Q:\Lambda^2\to\mathbb{F}$ be a pair of determinantally equivalent functions of class $\mathcal{D}$. Let $\mathcal{M}\subseteq\Lambda$ be a set such that $|\mathcal{M}|=4$. Then, either
	\begin{equation*}
		\text{every $3$-cycle in $\mathcal{M}$ belongs to Case 1 or every $3$-cycle in $\mathcal{M}$ belongs to Case 2.}
	\end{equation*}
\end{proposition}

Since the above proposition concerns a relationship between a pair of determinantally equivalent functions $K$ and $Q$ restricted on sets of cardinality four, in order to simplify notations we will assume throughout this section, without loss of generality, that $K$ and $Q$ have domain $\mathcal{M}^2$, where $\mathcal{M}\coloneqq\{1,2,3,4\}$; and we will be consistent with the labelling of the $4$-cycles and $3$-cycles in $\mathcal{M}$ given in \eqref{four_cycle_labellings} and \eqref{three_cycle_labellings}, respectively.

The following two lemmas are enough for establishing Proposition~\ref{prop_step_2_lambda_4_elements}.
\begin{lemma}
	\label{prop_step_2.1_lambda_4_elements}
	Let $\mathbb{F}$ be a field, and suppose that the pair of functions $K:\mathcal{M}^2\to\mathbb{F}$ and $Q:\mathcal{M}^2\to\mathbb{F}$ are determinantally equivalent and of class $\mathcal{D}$. If any three of the four $3$-cycles $p^{(1)}$, $p^{(2)}$, $p^{(3)}$, $p^{(4)}$ belong to the same Case, then all four belong to that Case.
\end{lemma}

\begin{lemma}
	\label{prop_step_2.1_lambda_4_elements2}
	Let $\mathbb{F}$ be a field, and suppose that the pair of functions $K:\mathcal{M}^2\to\mathbb{F}$ and $Q:\mathcal{M}^2\to\mathbb{F}$ are determinantally equivalent and of class $\mathcal{D}$. If any two of the four $3$-cycles $p^{(1)}$, $p^{(2)}$, $p^{(3)}$, $p^{(4)}$ belong to the same Case, then all four belong to that Case.
\end{lemma}
Proposition~\ref{prop_step_2_lambda_4_elements} is obtained by combining the two lemmas.
\begin{proof}[Proof of Proposition~\ref{prop_step_2_lambda_4_elements}]
	From the pigeonhole principle, we know that there will always be two $3$-cycles $p^{(i)}$ (out of the four from (\ref{three_cycle_labellings})) belonging to the same Case. The claim of the proposition then follows immediately from Lemma~\ref{prop_step_2.1_lambda_4_elements} and Lemma~\ref{prop_step_2.1_lambda_4_elements2}. \qed
\end{proof}	
It remains to prove Lemmas~\ref{prop_step_2.1_lambda_4_elements} and \ref{prop_step_2.1_lambda_4_elements2}. Since the action of the symmetric group $S_4$ is transitive on pairs of distinct undirected $3$-cycles on four vertices, we can actually assume, without loss of generality, which specific two $3$-cycles from \eqref{three_cycle_labellings} belong to Case 1. To be consistent with \cite{mantelos2026determinantally}, let us assume henceforth that it is $p^{(2)}$ and $p^{(4)}$ that belong to Case 1.

\textbf{\underline{Running assumptions for the remainder of Section~\ref{sec_proof_of_main_result}:}} $p^{(2)}$ and $p^{(4)}$ belong to Case 1.

\begin{proof}[Proof of Lemma~\ref{prop_step_2.1_lambda_4_elements}]
	Suppose, without loss of generality, that $p^{(3)}$ also belongs to Case 1. We therefore need to show that $p^{(1)}$ belongs to the same Case as $p^{(2)}$, $p^{(3)}$ and $p^{(4)}$.
	
	Suppose first that $K=0$ on at least one of the edges of $p^{(1)}$. We note that every edge of $p^{(1)}$ is an edge of exactly one of $p^{(2)}$, $p^{(3)}$ and $p^{(4)}$. Since each of the latter three $3$-cycles belongs to Case 1, it suffices to consider the case where $K=0$ on the edge $(1,2)$, that is, $K(1,2)=0$; in which case $K[p^{(1)}]=0$. Since $(1,2)$ is also an edge of $p^{(2)}$, our assumption that it belongs to Case 1 yields $Q(1,2)=0$; thus, $$Q[p^{(1)}]=0=K[p^{(1)}].$$ Therefore, $p^{(1)}$ belongs to Case 1 just like the other three $3$-cycles. The analogous argument shows that if $K=0$ on at least one of the edges of the cycle $p^{(1)}$ in reverse, then $$K'[p^{(1)}]=0=Q'[p^{(1)}],$$ which is to say, $p^{(1)}$ belongs to Case 1, once again.
	
	Let us now suppose that $K\neq 0$ on every edge (and reverse edge) of $p^{(1)}$; the setup is thus:
	\begin{figure}[H]
		\centering
		\begin{tikzpicture}
			\Vertex[label=$2$]{A} \Vertex[x=6,label=$3$]{B}, \Vertex[x=3,y=3,label=$1$]{C} \Vertex[x=3,y=1,label=$4$]{D}
			\Edge[Direct,color=black,bend=30,label=$\neg 0$](A)(C)
			\Edge[Direct,color=blue,label=$\neg 0$](C)(A)
			\Edge[Direct,color=black,bend=30,label=$\neg 0$](C)(B)
			\Edge[Direct,color=green,label=$\neg 0$](B)(C)
			\Edge[Direct,color=black,bend=30,label=$\neg 0$](B)(A)
			\Edge[Direct,color=red,label=$\neg 0$](A)(B)
			\Edge[Direct,color=blue,bend=-20](D)(C)
			\Edge[Direct,color=green,bend=-20](C)(D)
			\Edge[Direct,color=blue,bend=-10](A)(D)
			\Edge[Direct,color=red,bend=-10,](D)(A)
			\Edge[Direct,color=red,bend=-10](B)(D)
			\Edge[Direct,color=green,bend=-10](D)(B)
		\end{tikzpicture}
		\caption{In black is the $3$-cycle $p^{(1)}$ in reverse, in red is the $3$-cycle $\color{red} p^{(4)}$, in green is the $3$-cycle $\color{green} p^{(3)}$, and in blue is the $3$-cycle $\color{blue} p^{(2)}$.}
	\end{figure}
	If none of
	\begin{equation}
		\label{one_of_quantities}
		K(1,4),\enspace K(4,1),\enspace K(2,4),\enspace K(4,2),\enspace K(3,4),\enspace K(4,3) 
	\end{equation}
	are zero, then it was proven in Lemma 6.4 of \cite{mantelos2026determinantally} that $K[p^{(1)}]=Q[p^{(1)}]$, and so we are done. So let us suppose that at least one of the quantities from \eqref{one_of_quantities} is zero. Since $p^{(2)}$, $p^{(3)}$ and $p^{(4)}$ belong to the same Case, we can suppose without loss of generality that $K(4,1)=0$ (all the other cases can be treated using the same method). Using Lemma~\ref{prop_one_zero_implies_many_nonzeros}, this yields the following setup:
	\begin{figure}[H]
		\centering
		\begin{tikzpicture}
			\Vertex[label=$2$]{A} \Vertex[x=6,label=$3$]{B}, \Vertex[x=3,y=3,label=$1$]{C} \Vertex[x=3,y=1,label=$4$]{D}
			\Edge[Direct,color=black,bend=30,label=$\neg 0$](A)(C)
			\Edge[Direct,color=blue,label=$\neg 0$](C)(A)
			\Edge[Direct,color=black,bend=30,label=$\neg 0$](C)(B)
			\Edge[Direct,color=green,label=$\neg 0$](B)(C)
			\Edge[Direct,color=black,bend=30,label=$\neg 0$](B)(A)
			\Edge[Direct,color=red,label=$\neg 0$](A)(B)
			\Edge[Direct,color=blue,bend=-20,label=0](D)(C)
			\Edge[Direct,color=green,bend=-20](C)(D)
			\Edge[Direct,color=blue,bend=-10](A)(D)
			\Edge[Direct,color=red,bend=-10,label=$\neg 0$](D)(A)
			\Edge[Direct,color=red,bend=-10](B)(D)
			\Edge[Direct,color=green,bend=-10,label=$\neg 0$](D)(B)
		\end{tikzpicture}
		\caption{In black is the $3$-cycle $p^{(1)}$ in reverse, in red is the $3$-cycle $\color{red} p^{(4)}$, in green is the $3$-cycle $\color{green} p^{(3)}$, and in blue is the $3$-cycle $\color{blue} p^{(2)}$.}
	\end{figure}
	We first note that $K(4,1)=0$ implies $K[p^{(2)}]=0$. We proceed by cases: 
	\begin{enumerate}[(I)]
		\item 	$p^{(2)}$ belongs to \textit{both} Case 1 and Case 2;
		\item $p^{(2)}$ belongs \textit{only} to Case 1.
	\end{enumerate}	
	We start by considering instance (I). Lemma~\ref{lemma_3cycle_both_Cases} then dictates that 
	\begin{equation}
		\label{eqn_all_zero_apply}
		K(1,4)=K(4,1)=Q(1,4)=Q(4,1)=0\text{ and } K(2,4)\neq 0,\enspace Q(2,4)\neq 0.
	\end{equation}
	%Using Proposition~\ref{prop_one_zero_implies_many_nonzeros} we thus have the following setup:
	\begin{comment}
		\begin{figure}[H]
			\centering
			\begin{tikzpicture}
				\Vertex[label=$2$]{A} \Vertex[x=6,label=$3$]{B}, \Vertex[x=3,y=3,label=$1$]{C} \Vertex[x=3,y=1,label=$4$]{D}
				\Edge[Direct,color=black,bend=30,label=$\neg 0$](A)(C)
				\Edge[Direct,color=blue,label=$\neg 0$](C)(A)
				\Edge[Direct,color=black,bend=30,label=$\neg 0$](C)(B)
				\Edge[Direct,color=green,label=$\neg 0$](B)(C)
				\Edge[Direct,color=black,bend=30,label=$\neg 0$](B)(A)
				\Edge[Direct,color=red,label=$\neg 0$](A)(B)
				\Edge[Direct,color=blue,bend=-20,label=0](D)(C)
				\Edge[Direct,color=green,bend=-20,label=0](C)(D)
				\Edge[Direct,color=blue,bend=-10,label=$\neg 0$](A)(D)
				\Edge[Direct,color=red,bend=-10,label=$\neg 0$](D)(A)
				\Edge[Direct,color=red,bend=-10,label=$\neg 0$](B)(D)
				\Edge[Direct,color=green,bend=-10,label=$\neg 0$](D)(B)
			\end{tikzpicture}
			\caption{In black is the $3$-cycle $p^{(1)}$ in reverse, in red is the $3$-cycle $\color{red} p^{(4)}$, in green is the $3$-cycle $\color{green} p^{(3)}$, and in blue is the $3$-cycle $\color{blue} p^{(2)}$.}
		\end{figure}
	\end{comment}
	This, in turn, yields
	\begin{equation*}
		K[q^{[1]}]=K'[q^{[1]}]=0=Q[q^{[1]}]=Q'[q^{[1]}]\text{ and }K[q^{[3]}]=K'[q^{[3]}]=0=Q[q^{[3]}]=Q'[q^{[3]}].
	\end{equation*}
	By Lemma 6.6 from \cite{mantelos2026determinantally}, this implies
	\begin{equation*}
		K[q^{[2]}]+K'[q^{[2]}]=Q[q^{[2]}]+Q'[q^{[2]}].
	\end{equation*}
	By the second part of Lemma~\ref{lemma_det_equivalence_w_easy_consequence}, we also have
	\begin{equation*}
		K[q^{[2]}]\cdot K'[q^{[2]}]=Q[q^{[2]}]\cdot Q'[q^{[2]}].
	\end{equation*}
	Consequently, either
	\begin{equation}
		\label{eq_choice_1}
		K[q^{[2]}]=Q[q^{[2]}]
	\end{equation}
	or
	\begin{equation}
		\label{eq_choice_2}
		K[q^{[2]}]=Q'[q^{[2]}].
	\end{equation}
	By Lemma~\ref{prop_one_zero_implies_many_nonzeros}, \eqref{eqn_all_zero_apply} also implies $K(3,4)\neq 0$. Lemma 5.2 from \cite{mantelos2026determinantally} (with a simple relabelling) is therefore applicable in this particular setup and yields
	\begin{equation*}
		K[q^{[2]}]=\frac{K[p^{(1)}]K'[p^{(4)}]}{K(2,3)K(3,2)},\enspace Q[q^{[2]}]=\frac{Q[p^{(1)}]Q'[p^{(4)}]}{Q(2,3)Q(3,2)}\text{ and }Q'[q^{[2]}]=\frac{Q'[p^{(1)}]Q[p^{(4)}]}{Q(2,3)Q(3,2)}.
	\end{equation*}
	So, if \eqref{eq_choice_1} is true, then
	\begin{equation*}
		\frac{K[p^{(1)}]K'[p^{(4)}]}{K(2,3)K(3,2)}=\frac{Q[p^{(1)}]Q'[p^{(4)}]}{Q(2,3)Q(3,2)}.
	\end{equation*}
	By applying the second part of Lemma~\ref{lemma_det_equivalence_w_easy_consequence} and using the fact that $p^{(4)}$ belongs to Case 1, this yields $$K[p^{(1)}]=Q[p^{(1)}];$$ which is to say, $p^{(1)}$ belongs to Case 1, as required.
	
	If, on the other hand, \eqref{eq_choice_2} is true, then we have
	\begin{equation*}
		\frac{K[p^{(1)}]K'[p^{(4)}]}{K(2,3)K(3,2)}=\frac{Q'[p^{(1)}]Q[p^{(4)}]}{Q(2,3)Q(3,2)}.
	\end{equation*}
	Thus, by the second part of Lemma~\ref{lemma_det_equivalence_w_easy_consequence},
	\begin{equation*}
		K[p^{(1)}]K'[p^{(4)}]=Q'[p^{(1)}]Q[p^{(4)}].
	\end{equation*}
	And so if $p^{(1)}$ did not belong to Case 1 (which necessarily means that $p^{(1)}$ belongs to Case 2), then the first term from both the left-hand-side and right-hand-side cancels, and we get $$K'[p^{(4)}]=Q[p^{(4)}];$$ which is to say, $p^{(4)}$ belongs to Case 2.
	
	Notice how $p^{(3)}$ also belongs to Case 2. Indeed, from Figure 4 and \eqref{eqn_all_zero_apply} we see that $K[p^{(3)}]=0$ and $K\neq 0$ on two of its edges, namely $(3,1)$ and $(4,3)$, and its other edge $(1,4)$ satisfies \eqref{eqn_all_zero_apply}. This setup, by Lemma~\ref{lemma_3cycle_both_Cases}, indicates that $p^{(3)}$ belongs to \textit{both} Case 1 and Case 2.
	
	Since all four $3$-cycles, $p^{(1)}$, $p^{(2)}$, $p^{(3)}$, $p^{(4)}$,  belong to the same Case, namely Case 2, we are done.
	
	Finally, we consider instance (II). Lemma~\ref{lemma_3cycle_only_one_Case} (i) then implies $K(1,4)\neq 0$. We proceed by cases: $$K(2,4)=0\text{ or } K(2,4)\neq 0.$$
	Let us first assume $K(2,4)=0$. Lemma~\ref{prop_one_zero_implies_many_nonzeros} then brings about the following setup:
	\begin{figure}[H]
		\centering
		\begin{tikzpicture}
			\Vertex[label=$2$]{A} \Vertex[x=6,label=$3$]{B}, \Vertex[x=3,y=3,label=$1$]{C} \Vertex[x=3,y=1,label=$4$]{D}
			\Edge[Direct,color=black,bend=30,label=$\neg 0$](A)(C)
			\Edge[Direct,color=blue,label=$\neg 0$](C)(A)
			\Edge[Direct,color=black,bend=30,label=$\neg 0$](C)(B)
			\Edge[Direct,color=green,label=$\neg 0$](B)(C)
			\Edge[Direct,color=black,bend=30,label=$\neg 0$](B)(A)
			\Edge[Direct,color=red,label=$\neg 0$](A)(B)
			\Edge[Direct,color=blue,bend=-20,label=0](D)(C)
			\Edge[Direct,color=green,bend=-20,label=$\neg 0$](C)(D)
			\Edge[Direct,color=blue,bend=-10,label=0](A)(D)
			\Edge[Direct,color=red,bend=-10,label=$\neg 0$](D)(A)
			\Edge[Direct,color=red,bend=-10,label=$\neg 0$](B)(D)
			\Edge[Direct,color=green,bend=-10,label=$\neg 0$](D)(B)
		\end{tikzpicture}
		\caption{In black is the $3$-cycle $p^{(1)}$ in reverse, in red is the $3$-cycle $\color{red} p^{(4)}$, in green is the $3$-cycle $\color{green} p^{(3)}$, and in blue is the $3$-cycle $\color{blue} p^{(2)}$.}
	\end{figure}
	Since $(4,1)$ and $(2,4)$ are edges of $p^{(2)}$ -- a $3$-cycle assumed to belong to Case 1 -- such that $K(4,1)=0=K(2,4)$, it follows that $Q(4,1)=0=Q(2,4)$. Therefore,
	\begin{equation*}
		\label{eqn_want_use_1}
		K[q^{[1]}]=K[q^{[2]}]=K[q^{[3]}]=0=Q[q^{[3]}]=Q[q^{[2]}]=Q[q^{[1]}].
	\end{equation*}
	Moreover, by Lemma 5.2 from \cite{mantelos2026determinantally} (with a simple relabelling of vertices),
	\begin{equation*}
		K'[q^{[3]}]=\frac{K[p^{(4)}]K[p^{(3)}]}{K(3,4)K(4,3)}=\frac{Q[p^{(4)}]Q[p^{(3)}]}{Q(3,4)Q(4,3)}=Q'[q^{[3]}],
	\end{equation*}
	where we have made use of the fact that $p^{(3)}$ and $p^{(4)}$ belong to Case 1 in conjunction with the second part of Lemma~\ref{lemma_det_equivalence_w_easy_consequence}. 		
	
	Thus, by Lemma~6.6 from \cite{mantelos2026determinantally},
	\begin{equation}
		\label{eqn_here_use}
		K'[q^{[1]}]+K'[q^{[2]}]=Q'[q^{[1]}]+Q'[q^{[2]}].
	\end{equation}
	Again by Lemma~5.2 from \cite{mantelos2026determinantally} (with a simple relabelling of vertices),
	\begin{equation*}
		Q'[q^{[1]}]=\frac{Q'[p^{(1)}]Q[p^{(3)}]}{Q(1,3)Q(3,1)}\text{ and }K'[q^{[1]}]=\frac{K'[p^{(1)}]K[p^{(3)}]}{K(1,3)K(3,1)}=K'[p^{(1)}]\cdot \frac{Q[p^{(3)}]}{Q(1,3)Q(3,1)}
	\end{equation*}
	and
	\begin{equation*}
		Q'[q^{[2]}]=\frac{Q'[p^{(1)}]Q[p^{(4)}]}{Q(2,3)Q(3,2)}\text{ and }K'[q^{[2]}]=\frac{K'[p^{(1)}]K[p^{(4)}]}{K(2,3)K(3,2)}=K'[p^{(1)}]\cdot \frac{Q[p^{(4)}]}{Q(2,3)Q(3,2)},
	\end{equation*}
	where we have made use of the fact that the $3$-cycles $p^{(3)}$ and $p^{(4)}$ belong to Case 1. Thus, by \eqref{eqn_here_use},
	\begin{equation*}
		K'[p^{(1)}]\left(  \frac{Q[p^{(3)}]}{Q(1,3)Q(3,1)} +  \frac{Q[p^{(4)}]}{Q(2,3)Q(3,2)}\right) = Q'[p^{(1)}]\left(  \frac{Q[p^{(3)}]}{Q(1,3)Q(3,1)} +  \frac{Q[p^{(4)}]}{Q(2,3)Q(3,2)}\right).
	\end{equation*}
	It follows that $K'[p^{(1)}]=Q'[p^{(1)}]$, that is, $p^{(1)}$ belongs to Case 1, as required.
	
	To complete the proof, it remains is to explore the case where $K(2,4)\neq 0$. We thus have the following setup:
	\begin{figure}[H]
		\centering
		\begin{tikzpicture}
			\Vertex[label=$2$]{A} \Vertex[x=6,label=$3$]{B}, \Vertex[x=3,y=3,label=$1$]{C} \Vertex[x=3,y=1,label=$4$]{D}
			\Edge[Direct,color=black,bend=30,label=$\neg 0$](A)(C)
			\Edge[Direct,color=blue,label=$\neg 0$](C)(A)
			\Edge[Direct,color=black,bend=30,label=$\neg 0$](C)(B)
			\Edge[Direct,color=green,label=$\neg 0$](B)(C)
			\Edge[Direct,color=black,bend=30,label=$\neg 0$](B)(A)
			\Edge[Direct,color=red,label=$\neg 0$](A)(B)
			\Edge[Direct,color=blue,bend=-20,label=0](D)(C)
			\Edge[Direct,color=green,bend=-20,label=$\neg 0$](C)(D)
			\Edge[Direct,color=blue,bend=-10,label=$\neg 0$](A)(D)
			\Edge[Direct,color=red,bend=-10,label=$\neg 0$](D)(A)
			\Edge[Direct,color=red,bend=-10](B)(D)
			\Edge[Direct,color=green,bend=-10,label=$\neg 0$](D)(B)
		\end{tikzpicture}
		\caption{In black is the $3$-cycle $p^{(1)}$ in reverse, in red is the $3$-cycle $\color{red} p^{(4)}$, in green is the $3$-cycle $\color{green} p^{(3)}$, and in blue is the $3$-cycle $\color{blue} p^{(2)}$.}
	\end{figure}
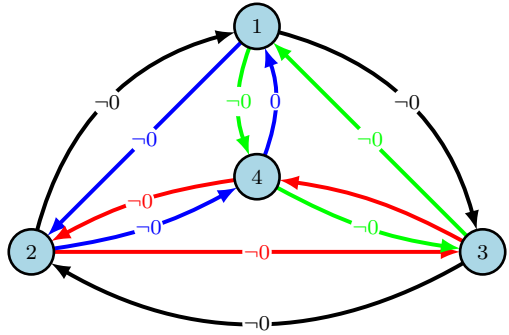
	Since $(2,4)$ is an edge of $p^{(2)}$ -- a $3$-cycle assumed to belong \textit{only} to Case 1 -- the fact that $K[p^{(2)}]=0$ and $K(2,4)\neq 0$ implies, by Lemma~\ref{lemma_3cycle_only_one_Case} (i), that $Q(2,4)\neq 0$ and $Q(4,2)\neq 0$ also. Utilizing Lemma~5.2 from \cite{mantelos2026determinantally} yet again (with a simple relabelling of the vertices) we obtain
	\begin{equation*}
		K'[q^{[1]}]=\frac{K'[p^{(2)}]K'[p^{(4)}]}{K(2,4)K(4,2)}=\frac{Q'[p^{(2)}]Q'[p^{(4)}]}{Q(2,4)Q(4,2)}=Q'[q^{[1]}]=Q'[p^{(1)}]\cdot\frac{Q[p^{(3)}]}{Q(1,3)Q(3,1)},
	\end{equation*}
	where the second equality follows from the fact that $p^{(2)}$ and $p^{(4)}$ belong to Case 1. But Lemma~5.2 from \cite{mantelos2026determinantally} also yields
	\begin{equation*}
		K'[q^{[1]}]=\frac{K'[p^{(1)}]K[p^{(3)}]}{K(1,3)K(3,1)}=K'[p^{(1)}]\cdot \frac{Q[p^{(3)}]}{Q(1,3)Q(3,1)},
	\end{equation*}
	where the second equality follows from the fact that $p^{(3)}$ belongs to Case 1. It follows that $K'[p^{(1)}]=Q'[p^{(1)}]$, that is, $p^{(1)}$ belongs to Case 1, as required. \qed

\end{proof}

\begin{proof}[Proof of Lemma~\ref{prop_step_2.1_lambda_4_elements2}]
	Thanks to Lemma~\ref{prop_step_2.1_lambda_4_elements}, if one of $p^{(1)}$,  $p^{(2)}$, $p^{(3)}$, $p^{(4)}$ belongs to both Case 1 and Case 2, we are done. So let us assume henceforth that this is not the case, which is to say, by Lemma~\ref{lemma_3cycle_in_both_cases_condition}, that for every $1\leq i \leq 4$,
	\begin{equation*}
		Q[p^{(i)}]\neq Q'[p^{(i)}].
	\end{equation*}
	Let us suppose, for a contradiction, that $p^{(1)}$ and $p^{(3)}$ belong to Case 2.
	
	The setting where $K$ and $Q$ are nowhere-zero except possibly on the set $\{(x,x):x\in\Lambda\}$ was already treated in \cite{mantelos2026determinantally} (and a contradiction was successfully obtained); so we do not assume this setting henceforth. 
	
	Because we are assuming that each $p^{(i)}$ belongs to only one Case, and more specifically, $p^{(2)}$ and $p^{(4)}$ to only Case 1, and $p^{(1)}$ and $p^{(3)}$ to only Case 2, Corollary~\ref{cor_impossible_zeroedge_shared_w_Case1_Case2} dictates that the only possible points in $\mathcal{M}^2$ on which $K$ can take on zero values are $(2,4)$, $(4,2)$, $(3,1)$ and $(1,3)$. Indeed, $(2,4)$ is the unique edge shared among $p^{(2)}$ and $(p^{(4)})'$, and $(3,1)$ that between $p^{(1)}$ and $p^{(3)}$. More specifically, by Corollary~\ref{cor_impossible_zeroedge_shared_w_Case1_Case2},
	\begin{enumerate}[(I)]
		\item $K(3,1)\neq0$, $K(1,3)\neq0$ and $K(2,4)= 0$, $K(4,2)\neq 0$, or 
		\item $K(3,1)=0$, $K(1,3)\neq0$ and $K(2,4)=0$, $K(4,2)\neq 0$, or
		\item $K(3,1)=0$, $K(1,3)\neq0$ and $K(2,4)\neq 0$, $K(4,2)\neq 0$, or
		\item $K(3,1)\neq 0$, $K(1,3)=0$ and $K(2,4)\neq0$, $K(4,2)\neq 0$, or
		\item $K(3,1)=0$, $K(1,3)\neq0$ and $K(2,4)\neq0$, $K(4,2)= 0$, or
		\item $K(3,1)\neq 0$, $K(1,3)=0$ and $K(2,4)\neq0$, $K(4,2)= 0$, or
		\item $K(3,1)\neq0$, $K(1,3)\neq0$ and $K(2,4)\neq 0$, $K(4,2)= 0$.
	\end{enumerate}
	By symmetry, it suffices to examine instances (I) and (II).
	
	Let us start by assuming (I). By Corollary~\ref{cor_impossible_zeroedge_shared_w_Case1_Case2} we then have the following setup.
	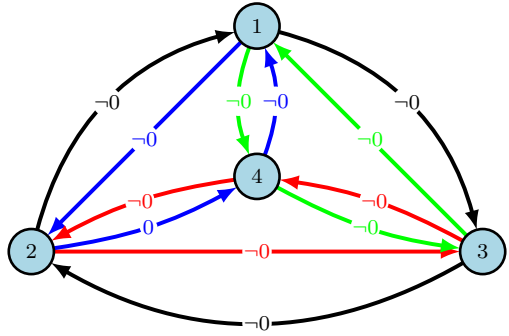
\begin{figure}[H]
		\centering
		\begin{tikzpicture}
			\Vertex[label=$2$]{A} \Vertex[x=6,label=$3$]{B}, \Vertex[x=3,y=3,label=$1$]{C} \Vertex[x=3,y=1,label=$4$]{D}
			\Edge[Direct,color=black,bend=30,label=$\neg 0$](A)(C)
			\Edge[Direct,color=blue,label=$\neg 0$](C)(A)
			\Edge[Direct,color=black,bend=30,label=$\neg 0$](C)(B)
			\Edge[Direct,color=green,label=$\neg 0$](B)(C)
			\Edge[Direct,color=black,bend=30,label=$\neg 0$](B)(A)
			\Edge[Direct,color=red,label=$\neg 0$](A)(B)
			\Edge[Direct,color=blue,bend=-20,label=$\neg 0$](D)(C)
			\Edge[Direct,color=green,bend=-20,label=$\neg 0$](C)(D)
			\Edge[Direct,color=blue,bend=-10,label=0](A)(D)
			\Edge[Direct,color=red,bend=-10,label=$\neg 0$](D)(A)
			\Edge[Direct,color=red,bend=-10,label=$\neg 0$](B)(D)
			\Edge[Direct,color=green,bend=-10,label=$\neg 0$](D)(B)
		\end{tikzpicture}
		\caption{In black is the $3$-cycle $p^{(1)}$ in reverse, in red is the $3$-cycle $\color{red} p^{(4)}$, in green is the $3$-cycle $\color{green} p^{(3)}$, and in blue is the $3$-cycle $\color{blue} p^{(2)}$.}
	\end{figure}
	Since $(2,4)$ is an edge of $p^{(2)}$ such that $K(2,4)=0$, and $p^{(2)}$ is a $3$-cycle belonging only to Case 1, Lemma~\ref{lemma_3cycle_only_one_Case} (i) dictates that $Q(2,4)=0$ also. It follows that
	\begin{equation*}
		K[q^{[2]}]=0=Q[q^{[2]}]\text{ and }K[q^{[3]}]=0=Q[q^{[3]}].
	\end{equation*}
	Moreover, by the usual application of Lemma~5.2 from \cite{mantelos2026determinantally} together with the fact that $p^{(1)}$ and $p^{(3)}$ both belong to Case 2,
	\begin{equation*}
		K[q^{[1]}]=\frac{K[p^{(1)}]K'[p^{(3)}]}{K(1,3)K(3,1)}=\frac{Q'[p^{(1)}]Q[p^{(3)}]}{Q(1,3)Q(3,1)}=Q'[q^{[1]}],
	\end{equation*}
	and similarly, $K'[q^{[1]}]=Q[q^{[1]}]$.
	
	Thus, by Lemma~6.6 from \cite{mantelos2026determinantally}, we have
	\begin{equation*}
		K'[q^{[2]}]+K'[q^{[3]}]=Q'[q^{[2]}]+Q'[q^{[3]}].
	\end{equation*}
	Furthermore, from the proof of Lemma~6.7 from \cite{mantelos2026determinantally}, 
	\begin{align*}
		K'[q^{[2]}]K'[q^{[3]}]&=K(1,3)K(3,1)\cdot K'[p^{(2)}]K[p^{(4)}]\\
		&=Q(1,3)Q(3,1)\cdot Q'[p^{(2)}]Q[p^{(4)}] &&\text{($p^{(2)}$ and $p^{(4)}$ belong to Case 1)}\\
		&=Q'[q^{[2]}]Q'[q^{[3]}].
	\end{align*}
	Therefore, either
	\begin{equation}
		\label{eqn_poss1}
		K'[q^{[2]}]=Q'[q^{[2]}]
	\end{equation}
	or
	\begin{equation}
		\label{eqn_poss2}
		K'[q^{[2]}]=Q'[q^{[3]}].
	\end{equation}
	Now, we also have, by the usual application of Lemma~5.2 from \cite{mantelos2026determinantally},
	\begin{align*}
		K'[q^{[2]}]&= \frac{K'[p^{(2)}]K'[p^{(3)}]}{K(1,4)K(4,1)}\\
		&= \frac{Q'[p^{(2)}]Q[p^{(3)}]}{K(1,4)K(4,1)} &&\text{($p^{(2)}$ belongs to Case 1, and $p^{(3)}$ to Case 2)}\\
		&= \frac{\cancel{Q(1,4)}Q(4,2)Q(2,1) Q(1,4)Q(4,3)Q(3,1)}{\cancel{Q(1,4)}Q(4,1)},
	\end{align*}
	Thus, if equation \eqref{eqn_poss1} is true, then
	\begin{equation*}
		\frac{\cancel{Q(4,2)}\cancel{Q(2,1)}Q(1,4)Q(4,3)Q(3,1)}{Q(4,1)}=Q(1,3)Q(3,4)\cancel{Q(4,2)}\cancel{Q(2,1)},
	\end{equation*}
	that is, $Q[p^{(3)}]=Q'[p^{(3)}]$, a contradiction.
	
	If, on the other hand, equation \eqref{eqn_poss2} is true, then
	\begin{equation*}
		\frac{\cancel{Q(4,2)}Q(2,1)\cancel{Q(1,4)}Q(4,3)\cancel{Q(3,1)}}{Q(4,1)}=\cancel{Q(1,4)}\cancel{Q(4,2)}Q(2,3)\cancel{Q(3,1)},
	\end{equation*}
	which is to say
	\begin{equation*}
		\begin{vmatrix}
			Q(2,3) & Q(2,1) \\
			Q(4,3) & Q(4,1)
		\end{vmatrix}=0,
	\end{equation*}
	a contradiction to our class $\mathcal{D}$ hypothesis.
	
	It remains to examine instance (II). By Corollary~\ref{cor_impossible_zeroedge_shared_w_Case1_Case2}, the setup is thus:
	\begin{figure}[H]
		\centering
		\begin{tikzpicture}
			\Vertex[label=$2$]{A} \Vertex[x=6,label=$3$]{B}, \Vertex[x=3,y=3,label=$1$]{C} \Vertex[x=3,y=1,label=$4$]{D}
			\Edge[Direct,color=black,bend=30,label=$\neg 0$](A)(C)
			\Edge[Direct,color=blue,label=$\neg 0$](C)(A)
			\Edge[Direct,color=black,bend=30,label=$\neg 0$](C)(B)
			\Edge[Direct,color=green,label=$0$](B)(C)
			\Edge[Direct,color=black,bend=30,label=$\neg 0$](B)(A)
			\Edge[Direct,color=red,label=$\neg 0$](A)(B)
			\Edge[Direct,color=blue,bend=-20,label=$\neg 0$](D)(C)
			\Edge[Direct,color=green,bend=-20,label=$\neg 0$](C)(D)
			\Edge[Direct,color=blue,bend=-10,label=0](A)(D)
			\Edge[Direct,color=red,bend=-10,label=$\neg 0$](D)(A)
			\Edge[Direct,color=red,bend=-10,label=$\neg 0$](B)(D)
			\Edge[Direct,color=green,bend=-10,label=$\neg 0$](D)(B)
		\end{tikzpicture}
		\caption{In black is the $3$-cycle $p^{(1)}$ in reverse, in red is the $3$-cycle $\color{red} p^{(4)}$, in green is the $3$-cycle $\color{green} p^{(3)}$, and in blue is the $3$-cycle $\color{blue} p^{(2)}$.}
	\end{figure}
	Since $(3,1)$ is an edge of $p^{(3)}$ -- a $3$-cycle assumed to belong only to  Case 2 -- the fact that $K(3,1)=0$ implies, by Lemma~\ref{lemma_3cycle_only_one_Case} (ii), that $Q(1,3)=0$ and $Q(3,1)\neq 0$.
	
	Since $(2,4)$ is an edge of $p^{(2)}$ -- a $3$-cycle assumed to belong only to Case 1 -- the fact that $K(2,4)=0$ implies, by Lemma~\ref{lemma_3cycle_only_one_Case} (i), that $Q(2,4)=0$.
	
	From $K(3,1)=0=Q(1,3)$ and $K(2,4)=0=Q(2,4)$, we get $$K[q^{[2]}]=K[q^{[3]}]=K'[q^{[3]}]=0=Q[q^{[3]}]=Q[q^{[2]}]=Q'[q^{[2]}].$$
	By Lemma~6.6 from \cite{mantelos2026determinantally}, this yields
	\begin{equation}
		\label{eqn_we_use_this_later}
		K[q^{[1]}]+K'[q^{[1]}]+K'[q^{[2]}]=Q[q^{[1]}]+Q'[q^{[1]}]+Q'[q^{[3]}].
	\end{equation}
	Since $K(1,3)\neq 0$ and $K(3,1)=0$, the second part of Lemma~\ref{graph_lemma_1} (i) yields
	\begin{align*}
		K'[q^{[1]}]&=K(3,4)K(4,3)\cdot K(4,1)K(1,4)\cdot\frac{K'[p^{(1)}]}{K'[p^{(3)}]} \\
		&= Q(3,4)Q(4,3)\cdot Q(4,1)Q(1,4)\cdot\frac{Q[p^{(1)}]}{Q[p^{(3)}]} &&\text{($p^{(1)}$ and $p^{(3)}$ belong to Case 2)}\\
		&=Q[q^{[1]}],
	\end{align*}
	where the last line follows from the first part of Lemma~\ref{graph_lemma_1} (i), since $Q(1,3)=0$ and $Q(3,1)\neq 0$.
	
	Moreover, since $K(2,4)=0$ and $K(4,2)\neq 0$, the first part of Lemma~\ref{graph_lemma_1} (ii) yields
	\begin{align*}
		K[q^{[1]}]&=K(4,1)K(1,4)\cdot K(1,2)K(2,1)\cdot \frac{K[p^{(4)}]}{K'[p^{(2)}]}\\
		&= Q(4,1)Q(1,4)\cdot Q(1,2)Q(2,1)\cdot \frac{Q[p^{(4)}]}{Q'[p^{(2)}]} &&\text{($p^{(2)}$ and $p^{(4)}$ belong to Case 2)}\\
		&=Q[q^{[1]}],
	\end{align*}
	where the last line follows from the first part of Lemma~\ref{graph_lemma_1} (ii), since $Q(2,4)=0$ and $Q(4,2)\neq 0$.
	
	By a simple application of the second part of Lemma~\ref{lemma_det_equivalence_w_easy_consequence}, we also have
	\begin{equation*}
		K[q^{[1]}]\cdot K'[q^{[1]}]=Q[q^{[1]}]\cdot Q'[q^{[1]}];
	\end{equation*} 
	from which $$Q[q^{[1]}]=Q'[q^{[1]}]$$
	follows. Consequently, equation \eqref{eqn_we_use_this_later} becomes
	\begin{equation*}
		K'[q^{[2]}]=Q'[q^{[3]}];
	\end{equation*}
	but this is precisely \eqref{eqn_poss2}, which had brought about a contradiction. \qed
\end{proof}

\bibliographystyle{plain}
\bibliography{DEF_structural_characterization_bib.bib}

@article{mantelos2026determinantally,
	title={Determinantally equivalent nonzero functions},
	author={Mantelos, Harry Sapranidis},
	journal={Discrete Mathematics},
	volume={349},
	number={6},
	pages={115021},
	year={2026},
	publisher={Elsevier}
}

@article{launay2021determinantal,
	title={Determinantal point processes for image processing},
	author={Launay, Claire and Desolneux, Agn{\`e}s and Galerne, Bruno},
	journal={SIAM Journal on Imaging Sciences},
	volume={14},
	number={1},
	pages={304--348},
	year={2021},
	publisher={SIAM}
}

@book{kulesza2012learning,
	title={Learning with determinantal point processes},
	author={Kulesza, John A},
	year={2012},
	publisher={University of Pennsylvania}
}

@article{loewy1986principal,
	title={Principal minors and diagonal similarity of matrices},
	author={Loewy, Raphael},
	journal={Linear algebra and its applications},
	volume={78},
	pages={23--64},
	year={1986},
	publisher={Elsevier}
}

@article{equiv_symm_kernels_for_dpps,
	title={Equivalent symmetric kernels of determinantal point processes},
	author={Stevens, Marco},
	journal={Random Matrices: Theory and Applications},
	volume={10},
	number={03},
	pages={2150027},
	year={2021},
	publisher={World Scientific}
}
\end{document}